\documentclass[11pt]{article}

\usepackage{fancyhdr}
\usepackage{amsfonts}
\usepackage{amsmath}
\usepackage{amssymb}
\usepackage{amsthm}
\usepackage{tikz}

\usepackage{amscd}
\usepackage{amstext}

\usepackage{accents}
\usepackage{tikz-cd}

\usepackage{mathrsfs}

\usepackage{authblk}

\usepackage{graphics}
\usepackage{graphicx}

\usepackage{graphics}
\usepackage{graphicx}

\newlength{\fixboxwidth}
\newcommand{\re}{\mathbb{R}}\newcommand{\N}{\mathbb{N}}
\newcommand{\zz}{\mathbb{Z}}\newcommand{\C}{\mathbb{C}}

\newcommand{\Z}{{\zz}^d}

\newcommand{\R}{{\re}^d}

\newcommand{\cs}{{\mathcal S}}

\newcommand{\ca}{{\mathcal A}}

\newcommand{\gf}{{\mathcal F}}
\newcommand{\cfi}{{\gf}^{-1}}

\newcommand{\supp}{{\rm supp \, }}

\newcommand{\bproof}{\begin{proof}}
\newcommand{\eproof}{\end{proof}}

\newcommand{\be}{\begin{equation}}
\newcommand{\ee}{\end{equation}}
\newcommand{\beq}{\begin{eqnarray}}
\newcommand{\beqq}{\begin{eqnarray*}}
\newcommand{\eeq}{\end{eqnarray}}
\newcommand{\eeqq}{\end{eqnarray*}}

\numberwithin{equation}{section}

\newtheorem{theorem}{Theorem}[section]
\newtheorem{definition}[theorem]{Definition}

\newtheorem{lemma}[theorem]{Lemma}
\newtheorem{proposition}[theorem]{Proposition}
\newtheorem{remark}[theorem]{Remark}

\begin{document}

\title{Isotropic and Dominating Mixed Lizorkin$-$Triebel Spaces -- a Comparison}

\author[a,b]{Van Kien Nguyen\thanks{E-mail: kien.nguyen@uni-jena.de \& kiennv@utc.edu.vn}}
\author[a]{Winfried Sickel\thanks{E-mail: winfried.sickel@uni-jena.de}}
\affil[a]{Friedrich-Schiller-University Jena, Ernst-Abbe-Platz 2, 07737 Jena, Germany}
\affil[b]{University of Transport and Communications, Dong Da, Hanoi, Vietnam}


\date{\today}

\maketitle
\begin{abstract}
We shall compare isotropic Lizorkin-Triebel spaces  with their counterparts of dominating mixed smoothness. 
\end{abstract}


\section{Introduction}


Let $t\in \N_0$ and $1< p< \infty$. The standard isotropic  Sobolev spaces is defined as
\beqq
W^t_p(\R)=\Big\{f\in L_p(\R): \|f|W^t_p(\R)\|=\sum_{|\bar{\alpha}|_1 \leq t}\|D^{\bar{\alpha}}f|L_p(\R)\|<\infty\Big\}\, .
\eeqq 
The Sobolev spaces of dominating mixed smoothness is given by 
\beqq
S^t_pW(\R)=\Big\{f\in L_p(\R): \|f|S^t_pW(\R)\|=\sum_{|\bar{\alpha}|_{\infty} \leq t}\|D^{\bar{\alpha}}f|L_p(\R)\|<\infty\Big\}.
\eeqq 
Here $\bar{\alpha}=(\alpha_1,...,\alpha_d)\in \N_0^d$, 
$|\bar{\alpha}|_1=\alpha_1+...+\alpha_d$ and $|\bar{\alpha}|_{\infty}=\max_{i=1,...,d}|\alpha_i|$. 
Obviously we have the chain of continuous embeddings
\beqq
W^{td}_p(\R) \hookrightarrow S^t_pW(\R) \hookrightarrow W^t_p(\R).
\eeqq
Also easy to see is the optimality of these embeddings in various directions. We will discuss this below.
These two types of Sobolev spaces  $W^t_p(\R)$ and $S^t_pW(\R)$ represent particular cases of corresponding scales of 
Bessel potential spaces (Sobolev spaces of fractional order $t$ of smoothness), 
denoted by 
$H^t_{p} (\R)$ (isotropic smoothness) and  $S^t_{p} H(\R)$ (dominating mixed smoothness).
Let $t\in \re$ and $1<p<\infty$. Then the space $H^t_{p}(\R)$ is the collection of all distributions  $f\in \cs'(\R)$ such that
 \beqq  
 \|f|H^t_{p}(\R)\|=\|\gf^{-1}\big[(1+|\xi|^2)^{\frac{t}{2}}\gf f(\xi)\big](\cdot)|L_p(\R)\| <\infty\, , 
 \eeqq
whereas  $S^t_{p}H(\R)$ is the collection of all $f\in \cs'(\R)$ such that
\beqq
\| f|S^t_{p}H(\R)\|=\Big\|\gf^{-1}\Big[\prod_{i=1}^{d}(1+\xi_i^2)^{\frac{t}{2}}\gf f(\xi)\Big](\cdot)\Big|L_p(\R)\Big\|  <\infty\,.
\eeqq
Here $\xi=(\xi_1,\ldots,\xi_d)\in \R$.
Indeed, if $t\in \N_0$ we have
\[
W^t_p(\R) = H^t_{p}(\R) \qquad \mbox{and}\qquad S^t_pW(\R) = S^t_{p} H(\R) 
\]
in the sense of equivalent norms.
In case $t>0$  Schmeisser \cite{Sc2}  stated  that
\be\label{ws-01}
H^{td}_{p}(\R)\hookrightarrow S^t_{p}H(\R)\hookrightarrow H^t_p(\R).
\ee
 In this paper we shall give a proof of \eqref{ws-01} and we shall  show the  optimality of these  assertions in the following directions:
\begin{itemize}
 \item Within all spaces   $S^{t_0}_{p_0}H (\R)$ satisfying $S^{t_0}_{p_0}H (\R) \hookrightarrow H^{t}_{p} (\R)$ 
the class $S^{t}_{p}H (\R)$ is the largest one.
\item
 Within all spaces   $H^{t_0}_{p_0} (\R)$ satisfying
$ H^{t_0}_{p_0} (\R) \hookrightarrow S^{t}_{p}H(\R) $
the class $H^{td}_{p} (\R)$ is the largest one.
\item
Within all spaces   $S^{t_0}_{p_0}H (\R)$ satisfying
$H^{td}_{2} (\R) \hookrightarrow S^{t_0}_{p_0} H(\R) $
the class $S^{t}_{p} H(\R)$ is the smallest one.
\end{itemize}
In what follows we shall go one step further.
Isotropic Sobolev spaces $H^t_p(\R)$ and Sobolev spaces of dominating mixed smoothness $ S^t_pH(\R)$ 
of fractional order $t$ are contained as special cases in the scales of  isotropic 
Lizorkin-Triebel  spaces $F^t_{p,q}(\R)$ and Lizorkin-Triebel spaces of dominating mixed smoothness $S^t_{p,q}F(\R)$.
It is well-known that
\[
H^t_p(\R)=F^t_{p,2}(\R) \qquad  \mbox{and} \qquad S^t_p H(\R) = S^t_{p,2}F(\R), \qquad 1<p<\infty\, , \quad  t\in \re, 
\]
in the sense of equivalent norm, see \cite[Theorem 2.5.6]{Tr83} and  \cite[Theorem 2.3.1]{ST}. 
In this paper we address the question under  which conditions on $t,p,q$ the embeddings 
\beqq
 F^{td}_{p,q}(\R)\hookrightarrow S^t_{p,q} F(\R) \hookrightarrow F^t_{p,q}(\R)
 \eeqq
hold true. In addition we shall discuss the optimality of these embeddings in various directions.

Nowadays isotropic Lizorkin-Triebel spaces represent  a well accepted regularity notion in various fields of mathematics.
Lizorkin-Triebel spaces of dominating mixed smoothness, in particular the scale $S^t_p H (\R)$,
 are of increasing importance in approximation theory and information-based complexity.  
As  special cases, the scale $S^t_p H (\R) = S^t_{p,2} F(\R)$  contains the tensor products of the univariate  spaces $H^t_{p}(\re)$, 
and the scale $S^t_{p,p} F(\R)$  contains the tensor products of the univariate  spaces $F^t_{p,p}(\re)$, 
see \cite{SU09,SU10}.
It is the main aim of this paper to give a detailed comparison of these different extensions of univariate Lizorkin-Triebel  spaces
into the multi-dimensional situation.
\\

The paper is organized as follows. Section \ref{sec:def} is devoted to the definition and some basic properties of the function
spaces under consideration. 
Our main results are stated in Section \ref{sec:main}.
Almost all proofs are concentrated in Section \ref{sec:proof}. 
In Subsection \ref{sec:prel} we collect the required tools from Fourier analysis,
especially some vector-valued Fourier multiplier assertions.
The next Subsection \ref{cinterpol} is devoted to complex interpolation.
Dual spaces are discussed in Subsection \ref{dualspaces}.
Finally, we collect families of test functions in Subsection \ref{test}.


\subsection*{Notation}


As usual, $\N$ denotes the natural numbers, $\N_0 := \N \cup \{0\}$,
$\zz$ the integers and
$\re$ the real numbers, $\C$ refers to the complex numbers. For a real number $a$ we put $a_+ := \max(a,0)$. 
The letter $d$ is always reserved for the underlying dimension in $\R , \Z $. We denote by
 $\langle x,y\rangle$ or $x\cdot y$ the usual Euclidean inner product in $\R $ or $\C^d$. By $x\diamond y$ we mean
\[
x\diamond y =(x_1y_1,...\,, x_dy_d)\in \R\,.
\]
If $\bar{k}= (k_1, \ldots  , k_d)\in \N_0^d$, then we put
\[
|\bar{k}|_1 := k_1 + \ldots \, + k_d\, \quad
 \text{ and }\quad |\bar{k}|_\infty:= \max_{j=1, \ldots , d} \, |k_j|  .
 \]
For $\bar{k} \in \N_0^d$ and $a>0$ we write $a^{\bar{k}} := (a^{{k_1}}, \ldots , a^{{k_d}})$.
By  $C, C_1,C_2, \,  \ldots $ we  denote   positive
constants which are independent of the main parameters involved but
whose values may differ from line to line. The symbol 
$A \asymp B$ means that there exist positive constants $C_1$ and $C_2$ such that $C_1\, A\leq B\leq C_2\, A.$\\
Let $X$ and $Y$ be  two quasi-Banach spaces. Then $X \hookrightarrow Y$ indicates that the embedding is continuous.
 Let $L_p(\R)$, $0 < p\le\infty$, be the space of all functions 
$f:\R\to\C$ such that 
\[
\|f|L_p(\R)\| := \Big(\int\limits_{\R} |f(x)|^p d x \Big)^{1/p} < \infty
\]
with the usual modification if $p=\infty$. 
By $C_0^\infty (\R)$ the set of all compactly supported infinitely differentiable functions $f:\R \to \C$ is denoted.
Let $\cs (\R)$ be the Schwartz space of all complex-valued rapidly decreasing infinitely differentiable  functions on $\R$. 
The topological dual, the class of tempered distributions, is denoted by $\cs'(\R)$ (equipped with the weak topology).
The Fourier transform on $\cs(\R)$ is given by 
\[
\gf \varphi (\xi) = (2\pi)^{-d/2} \int\limits_{\R} \, e^{-ix \xi}\, \varphi (x)\, dx \, , \qquad \xi \in \R\, .
\]
The inverse transformation is denoted by $\cfi $.
We use both notations also for the transformations defined on $\cs'(\R)$.
Let $0<p,q\leq \infty$. For a sequence of complex-valued functions $\{f_{\bar{k}}\}_{\bar{k}\in \N_0^d}$ on $\R$, we put
\beqq
\|f_{\bar{k}}|L_p(\ell_q)\|=\Big\|\Big( \sum_{\bar{k}\in\N_0^d} |f_{\bar{k}}|^q \Big)^{1/q}\Big|L_p(\R)\Big\|.
\eeqq


\section{Spaces of isotropic and dominating mixed smoothness }\label{sec:def}


The isotropic spaces $F^t_{p,q} (\R)$ are invariant under rotations, the spaces 
$S^t_{p,q}F(\R)$ are not invariant under rotations.
Both properties have been known for a long time and are well reflected 
by trace assertions on hyperplanes, see, e.g., Triebel
 \cite[2.7]{Tr83} (isotropic spaces) and 
Triebel \cite{Tr89},  Vyb\'\i ral \cite{Vy1}, Vyb\'\i ral and S. \cite{VS} (dominating mixed smoothness).


\subsection{Isotropic Besov-Lizorkin-Triebel spaces}


For us it will be convenient to introduce Lizorkin-Triebel and Besov spaces simultaneously.
\\
Let $\phi_0\in \cs(\R)$ be a non-negative function such that $\phi_0(x)=1$ if $|x|\leq 1$ and 
$\phi_0 (x)=0$ if $|x|\geq \frac{3}{2}$. For $j\in \N$ we define 
\beqq
\phi_j(x):=\phi_0(2^{-j}x)-\phi_0(2^{-j+1}x)\, , \qquad x \in \R\, .
\eeqq

\begin{definition} Let $0< p,q \le\infty$ and $t \in \re$.
\\
{\rm (i)}
The Besov space $B^t_{p,q}(\R)$ is then the collection of all
tempered distributions $f\in \mathcal{S}'(\R)$ such that
\[
\|\, f \, |B^t_{p,q}(\R)\|^{\phi}:= \Big(\sum\limits_{j=0}^{\infty}
2^{jtq}\, \|\, \cfi [\phi_j \gf f] (\, \cdot\, )\,
|L_p(\R)\|^q\Big)^{1/q}
\]
is finite. 
\\
{\rm (ii)} Let $p< \infty$.
The Triebel-Lizorkin space $F^t_{p,q}(\R)$ is then the collection of all
tempered distributions $f\in \mathcal{S}'(\R)$ such that
\[
\|\, f \, |F^t_{p,q}(\R)\|^{\phi}:= \Big\| \Big(\sum\limits_{j=0}^{\infty}
2^{jtq}\, |\, \cfi [\phi_j \gf f] (\, \cdot\, )\,|^q\Big)^{1/q}\, \Big| L_p(\R)\Big\|
\]
is finite. 
\end{definition}

\begin{remark}
\rm
Lizorkin-Triebel spaces are discussed in various monographs, let us refer, e.g., to 
\cite{Tr83}, \cite{Tr92}, \cite{Tr06} and \cite{BIN96}.
They are quasi-Banach spaces (Banach spaces if $\min(p,q)\ge 1$) and 
they do not depend on the chosen generator $\phi_0$ of the smooth dyadic decomposition
(in the sense of equivalent quasi-norms).
We call them isotropic because they are invariant under rotations.
Characterizations in terms of differences can be found at various places, see, e.g., \cite[2.5]{Tr83},  
\cite[3.5]{Tr92} or \cite[Sect.~28]{BIN96}.
\end{remark}

Many times we will work with the following  equivalent quasi-norm. Let $\psi_0\in \cs(\R)$ such that for $x \in \R$
 \beqq
 \psi_0(x)=1 \quad \text{\ if\ } \quad \sup_{i=1,...,d}|x_i|\leq 1
\qquad \text{and}\qquad  \psi_0(x)=0 \quad \text{ if }\quad \sup_{i=1,...,d}|x_i|\geq \frac{3}{2}\, .
\eeqq
 For $j\in \N$, we define 
 \beqq
 \psi_j(x):=\psi_0(2^{-j}x)-\psi_0(2^{-j+1}x).
 \eeqq 
 Then we have
 \beqq
 \supp\psi_j \subset \Big\{x:\ \sup_{i=1,...,d}|x_i|\leq 3\, \cdot \, 2^{j-1}\Big\}\backslash\Big\{x: \ \sup_{i=1,...,d}|x_i|\leq 2^{j-1}\Big\}.
 \eeqq

 \begin{proposition}\label{equi}
Let $t \in \re$, $0<p<\infty$ and $0<q\leq \infty$. Then $F^{t}_{p,q}(\R)$ is the collection of all
         $f\in \cs'(\R)$ such that
\beqq
            \|f|F^{t}_{p,q}(\R)\|^{\psi}  =
            \bigg\|\bigg(\sum\limits_{j=0}^{\infty}2^{jtq}
            |\gf^{-1}[\psi_{j}\gf f](\cdot)|^q\bigg)^{1/q}
            \bigg|L_p(\R)\bigg\|
\eeqq
         is finite. The quasi-norms $\|f|F^{t}_{p,q}(\R)\|^{\psi}$ and $\|f|F^{t}_{p,q}(\R)\|^{\phi}$ are equivalent.
 \end{proposition}

The equivalence of these quasi-norms  has been proved in \cite[Proposition 2.3.2]{Tr83} (in a much more general framework).
Proposition \ref{equi} also holds true for Besov spaces (with the respective quasi-norms). 
From now on  we will work with  the $\psi-$norm  and therefore we  write $\|f|F^t_{p,q}(\R)\| $
 instead of $\|f|F^t_{p,q}(\R)\|^{\psi}$.


\subsection{Besov-Lizorkin-Triebel spaces of dominating mixed smoothness}


Next we will give the definitions of Besov and Lizorkin-Triebel spaces of 
dominating mixed smoothness. 
We start with a smooth dyadic decomposition on $\re$ and afterwards we shall take its $d$-fold tensor product.
More exactly, let  $\varphi_0 \in C_0^{\infty}({\re})$ satisfying
$\varphi_0(\xi) = 1$ on $[-1,1]$ and $\supp\varphi \subset [-\frac{3}{2},\frac{3}{2}]$. For $j\in \N$ we define
\be\label{varphi}
         \varphi_j(\xi) = \varphi_0(2^{-j}\xi)-\varphi_0(2^{-j+1}\xi)\, ,\quad \xi \in \re\, .
\ee
Now we turn to tensor products.
   For $\bar{k} = (k_1,\, \ldots \, ,k_d) \in {\N}_0^d$ we put
\be\label{decom}
        \varphi_{\bar{k}}(x) := \varphi_{k_1}(x_1)\cdot \, \ldots \, \cdot
         \varphi_{k_d}(x_d)\, ,\qquad x\in \R.
\ee
This construction results in a smooth dyadic decomposition of unity $\{ \varphi_{\bar{k}}\}_{\bar{k} \in \N_0^d}$ on $\R$.

\begin{definition} Let $0<p,q \le \infty$ and $t \in\re$.
\\
{\rm (i)}
The Besov space of dominating mixed smoothness $ S^{t}_{p,q}B(\re^d)$ is the
collection of all tempered distributions $f \in \mathcal{S}'(\R)$
such that
\[
 \|\, f \, |S^{t}_{p,q}B(\R)\| :=
\Big(\sum\limits_{\bar{k}\in \N_0^d} 2^{|\bar{k}|_1 t q}\, \|\, \cfi[\varphi_{\bar{k}}\, \gf f](\, \cdot \, )
|L_p(\re^d)\|^q\Big)^{1/q}
\]
is finite. 
\\
{\rm (ii)} Let $0 < p< \infty$.
The Lizorkin-Triebel space of dominating mixed smoothness $ S^{t}_{p,q}F(\re^d)$ is the
collection of all tempered distributions $f \in \mathcal{S}'(\R)$
such that
\be\label{norm2}
 \|\, f \, |S^{t}_{p,q}F(\R)\| :=
\Big\| \Big(\sum\limits_{\bar{k}\in \N_0^d} 2^{|\bar{k}|_1 t q}\, 
|\, \cfi[\varphi_{\bar{k}}\, \gf f](\, \cdot \, )|^q \Big)^{1/q} \Big|L_p(\re^d)\Big\|
\ee
is finite. 
\end{definition}

\begin{remark}\label{blabla}
\rm
{\rm (i)} For $d=1$ we have $
 S^{t}_{p,q} A(\re) = A^t_{p,q}(\re)\,   $, $A\in \{ B, F\}$. \\
{\rm (ii)} Lizorkin-Triebel spaces of dominating mixed smoothness have a cross-quasi-norm, i.e., 
if $f_i \in F^t_{p,q}(\re)$, $ i=1, \ldots \, , d$, then it follows 
\[
 f(x) = \prod_{i=1}^d f_i (x_i)  \in S^t_{p,q}F(\R) \qquad\text{and}\qquad  
 \| \, f \, | S^t_{p,q}F(\R)\| = \prod_{i=1}^d \|\, f_i \, |F^t_{p,q} (\re)\| \, .
\]
\end{remark}
{~}\\
Of certain use for us will be the following Nikol'skij representation for 
Lizorkin-Triebel spaces of dominating mixed smoothness.

\begin{proposition}\label{inf}
Let  $1<p<\infty$, $1 \le   q \leq \infty$ and $t \in \re$. Let further 
$\{\varphi_{\bar{k}} \}_{\bar{k} \in \N_0^d}$ be the above system. Then 
the space $S^t_{p,q}F(\R)$ is a collection of all $f\in \cs'(\R)$ such that there exists a 
sequence $\{f_{\bar{k}}\}_{\bar{k} \in \N_0^d} \subset L_p(\R)$ satisfying
\be\label{equ-1-1}
 f=\sum_{\bar{k} \in \N_0^d}\gf^{-1}\varphi_{\bar{k}}\gf f_{\bar{k}} \ \ \text{in} \ \cs'(\R) 
\quad \text{ and}\quad \|2^{t|\bar{k}|_1}f_{\bar{k}}|L_{p}(\ell_q)\| <\infty\,.
\ee
The norm
\beqq
\| f| S^t_{p,q}F(\R)\|^* : =\inf \|2^{t|\bar{k}|_1}f_{\bar{k}}|L_{p}(\ell_q)\|
\eeqq
is equivalent to the norm in \eqref{norm2}. Here the infimum is taken over all admissible representations in \eqref{equ-1-1}.
\end{proposition}

\begin{proof} 
{\em Step 1.}
Let $\{\varphi_{j}\}_{j=0}^{\infty}$ be the system given in \eqref{varphi}. We put 
\be\label{eins}
\tilde{\varphi}_{j}:=\varphi_{j-1}+\varphi_j+\varphi_{j+1}\, , \qquad j\in \N_0 \, , 
\ee
 with $\varphi_{-1}\equiv 0$. If $\bar{k} \in \N_0^d$ we define $\tilde{\varphi}_{\bar{k}}
:=\tilde{\varphi}_{k_1} \otimes \cdots \otimes\tilde{\varphi}_{k_d}$. For $f\in S^t_{p,q}F(\R)$ 
we choose $f_{\bar{k}}= \gf^{-1}\tilde{\varphi}_{\bar{k}}\gf f\,.$
It follows from $\sum_{\bar{k}\in \Z} {\varphi}_{\bar{k}} = 1 $ for all $x \in \R$ and 
$\tilde{\varphi}_{j} (x) = 1$ if $x \in \supp {\varphi}_{j}$
that 
\[
\sum_{\bar{k} \in \N_0^d}\gf^{-1}[\varphi_{\bar{k}}\gf f_{\bar{k}}] = \sum_{\bar{k} \in \N_0^d}\gf^{-1} 
[\varphi_{\bar{k}} \, \tilde{\varphi}_{\bar{k}} \,  \gf f]
= f \, .
\]
Hence
\beqq
\begin{split}
\| f| S^t_{p,q}F(\R)\|^*&\leq\, \|2^{t|\bar{k}|_1}f_{\bar{k}}|L_{p}(\ell_q)\| \\
 &
= 
\big\|2^{t|\bar{k}|_1}\gf^{-1}\tilde{\varphi}_{\bar{k}}\gf f|L_{p}(\ell_q)\big\|
\leq\, 3^{d}\, \big\|2^{t|\bar{k}|_1}\gf^{-1}\varphi_{\bar{k}}\gf f|L_{p}(\ell_q)\big\|\,.
\end{split}
\eeqq
{\em Step 2.} Assume that $f$ can be represented  as in \eqref{equ-1-1}. 
We put $\varphi_{\bar{k}}\equiv 0$ if $\min_{i=1,...,d} k_i<0$. 
Then we obtain
\beqq
\gf^{-1}\varphi_{\bar{k}}\gf f
= \gf^{-1}\Big(\varphi_{\bar{k}} \sum_{\bar{\ell}\in \{-1,0,1\}^d}
\varphi_{\bar{k}+\bar{\ell}}\gf f_{\bar{k}+\bar{\ell}}\Big) \, .
\eeqq 
Applying  Lemma \ref{mul1} we get 
\beqq
\|2^{t|\bar{k}|_1}\gf^{-1}\varphi_{\bar{k}}\gf f|L_{p}(\ell_q)\| 
&\leq & c_1\,\Big\|2^{t|\bar{k}|_1}\gf^{-1}\Big( \sum_{ \bar{\ell}\in \{-1,0,1\}^d}
\varphi_{\bar{k}+\bar{\ell}}\gf f_{\bar{k}+\bar{\ell}}\Big)\Big|L_{p}(\ell_q)\Big\|
\\
&\leq  & c_2 \, \|2^{t|\bar{k}|_1}\gf^{-1}  \varphi_{\bar{k}}\gf f_{\bar{k}}|L_{p}(\ell_q)\|\, .
\eeqq
To continue we split $\sum_{\bar{k}}$ into several parts. Observe that
\be\label{zerl}
\sum_{\bar{k}\in \N_0^d} |2^{t|\bar{k}|_1}\gf^{-1}  [\varphi_{\bar{k}}\gf f_{\bar{k}}]|^q  
=  \sum_{e\subset \{1,...,d\}}\sum_{k_i\geq 1, i\in e\atop k_j=0, j\not \in e} 
|2^{t|\bar{k}|_1}\gf^{-1} [ \varphi_{\bar{k}}\gf f_{\bar{k}}]|^q.
\ee
Proposition \ref{four-1} can be applied to each subsum. 
This yields
\[
\Big\|2^{t|\bar{k}|_1}\gf^{-1} \varphi_{\bar{k}}\gf f_{\bar{k}}\, \Big|L_{p}(\ell_q)\Big\|
\leq c_3\, \|2^{t|\bar{k}|_1}  f_{\bar{k}}|L_{p}(\ell_q)\|
\]
with a constant $c_3$ independent of $f$.
The  proof is complete.
\end{proof}


\section{The main results}\label{sec:main}


As mentioned in the Introduction we will split our considerations into two cases: 
\begin{itemize}
 \item $S^{t}_{p,q}F(\R)\hookrightarrow F^{t}_{p,q}(\R)$;
\item $F^{td}_{p,q}(\R)\hookrightarrow S^{t}_{p,q}F(\R)$.
\end{itemize}


\subsection{The embedding of dominating mixed spaces into isotropic spaces}\label{SF-F}


The first of our main results reads as follows.

\begin{theorem}\label{main1}
Let $d\geq 2$, $0<p< \infty$ and $0<q\leq \infty$ and $t\in \re$. Then we have $$S^{t}_{p,q}F(\R)\hookrightarrow F^{t}_{p,q}(\R)$$
 if one of the following conditions is satisfied: 
\begin{itemize}
\item $t>0$;
\item $t=0$, $1<p<\infty$ and $0 <q\leq 2$;
\item $t=0$, $0 < p\leq 1$ and $0 < q< 2$;
\end{itemize}
\end{theorem}

\begin{remark}\rm
We recall that $S^{0}_{p,2}F(\R)= F^{0}_{p,2}(\R)=L_p(\R)$, $1<p<\infty$, in the sense of equivalent norms. 
This is a consequence of certain Littlewood-Paley assertions, see Nikol'skij \cite[1.5.6]{Ni}. 
This identity does not extend to $p=1$. Here we conjecture
\[
S^{0}_{1,2}F(\R) \hookrightarrow  F^{0}_{1,2}(\R) \hookrightarrow L_1(\R) \, .
\]
\end{remark}

\begin{theorem}\label{1<p,q<vc} 
Let $d\geq 2$, $1<p<\infty$, $1\leq q\leq \infty$ and $t\in \re$. Then
$$S^{t}_{p,q}F(\R)\hookrightarrow F^{t}_{p,q}(\R)$$
if and only if either $t>0$ or $t=0$ and $1\leq  q\leq 2$.
\end{theorem}

In addition we have the following. 

\begin{proposition}\label{pro:1}Let $d\geq 2$ and $t<0$.\\
{\rm (i)} If $1<p<\infty$ and $1\leq q\leq \infty$, then
$
F^{t}_{p,q}(\R)\hookrightarrow S^t_{p,q} F(\R) 
$\,.\\
{\rm (ii)} If $0<p<1$, $0<q\leq \infty$, then $F^{t}_{p,q}(\R)$ and $ S^t_{p,q} F(\R)$ are not comparable.
\end{proposition}

We summarize  the relation between $F^t
_{p,q}(\R)$ and $S^t
_{p,q}F(\R)$ ($1\leq q\leq \infty$) in the following figure.
$$
\begin{tikzpicture}
\draw[->, ](0,0) -- (8,0);
\draw[->, ] (0,-2 ) -- (0,2 );
\draw (4,-2 )-- (4,0);
\draw[->, ] (8,1.6) -- (7,0.1);
\node [] at (8,1.9) {{\rm critical line}};

\node[below] at (-0.2,0) {$0$};
\node [left] at (8,-0.5) {$\frac{1}{p}$};
\node [left] at (0,2) {$t$};
\node [] at (4,1 ) {$ S^{t}_{p,q}F(\R)  \hookrightarrow F^t_{p,q}(\R)$};

\node [] at (2,-1.4) {{\small  $ F^{t}_{p,q}(\R) \hookrightarrow S^{t}_{p,q}F(\R)$}};

\node [] at (5.8,-1) {{\rm not comparable}};
\fill (4,0) circle (2pt);
\fill (0,0) circle (2pt);
\node [left] at (4,-0.3) {1};
\node [right] at (-0.2,-2.8) {Figure 1. Comparison of $S^{t}_{p,q}F(\R)$ and $F^{t}_{p,q}(\R)$};
\path[draw, line width=1.5pt](0,0) -- (7.5,0);
\end{tikzpicture}
$$
These above embeddings are optimal in the following sense.

\begin{theorem}\label{tri-li5}
 Let $d\geq 2$, $0<p_0,p<\infty$, $0< q_0,q \le  \infty$ and $t_0,t\in \re$. 
 Let $p,q$ and $t$ be fixed.
 Within all spaces   $S^{t_0}_{p_0,q_0}F (\R)$ satisfying
\beqq
S^{t_0}_{p_0,q_0}F (\R) \hookrightarrow F^{t}_{p,q} (\R) 
\eeqq
the class $S^{t}_{p,q}F (\R)$ is the largest.
\end{theorem}

\begin{remark} \rm
Note that,  within all spaces   $F^{t_0}_{p_0,q_0} (\R)$ satisfying $
S^{t}_{p,q}F (\R) \hookrightarrow F^{t_0}_{p_0,q_0} (\R) 
$
the class $F^{t }_{p ,q } (\R)$ is not the smallest one. There is no smallest space with this respect. 
We may follow the arguments we have used in the context of the analogous problem for Besov spaces, see \cite{KS}. 
From Theorem \ref{main1} we have $S^{2}_{1,2}F (\R) \hookrightarrow F^2_{1,2} (\R) $.
On the other hand, a Sobolev-type embedding and Theorem \ref{main1} imply
\[
 S^{2}_{1,2}F (\R) \hookrightarrow S^{3/2}_{2,2}F (\R) \hookrightarrow F^{3/2}_{2,2} (\R) \, .
\]
However, for $d \ge 2$ the spaces  $F^2_{1,2} (\R)$ and $ F^{3/2}_{2,2} (\R)$ are not comparable. 
\end{remark}


\subsection{The embedding of isotropic spaces into dominating mixed spaces}\label{F-SF}


\begin{theorem}\label{main2}
Let $d\geq 2$, $0<p< \infty$, $0 < q \leq \infty$ and $t\in \re$. 
Then we have $$F^{td}_{p,q}(\R)\hookrightarrow S^{t}_{p,q}F(\R)$$ if  one 
of the following conditions is satisfied:
\begin{itemize}
 \item $t>\big(\frac{1}{\min(p,q)}-1\big)_+ $ and $0<q<\infty$; 
 \item $t=0$, $1 <p< \infty$ and $2\leq q\le \infty$.
\end{itemize}
\end{theorem}
\begin{remark}\rm The proof below is a bit more general than stated in Theorem \ref{main2}. Let  $q=\infty$.
Then we shall prove that  $F^{td}_{p,\infty}(\R)\hookrightarrow S^{t}_{p,\infty}F(\R)$ 
if either $p>1$ and $t>0$ or $0<p \le 1$ and $t>1/p$.  
For further comments, see Remark \ref{ga-pe}.
\end{remark}

By using a similar argument as in proof of  Theorem \ref{1<p,q<vc} we  conclude the following.

\begin{theorem}\label{1<p,q<vc-2} Let $d\geq 2$, $1<p<\infty$, $1\leq q\leq \infty$ and $t\in \re$. Then
$$F^{td}_{p,q} \hookrightarrow S^{t}_{p,q}F(\R)$$
if and only if either $t>0$ or $t=0$ and $q\geq  2$.
\end{theorem}

In addition we have the following supplement.

\begin{proposition}\label{pro:2} Let $d\geq 2$.\\
{\rm (i)}  Let $0<p<1$, $0<q\leq \infty$ and $0< t\leq \frac{1}{p}-1$. Then $S^{t}_{p,q}F(\R)$ and $F^{td}_{p,q} (\R) $ are not comparable.\\
{\rm (ii)} Let $0<p<\infty$, $0<q\leq \infty$ and $t<0$. Then $S^{t}_{p,q}F(\R)\hookrightarrow F^{td}_{p,q} (\R) $ follows. 
\end{proposition}



We summarize  the relation between $F^{td} _{p,q}(\R)$ and $S^t _{p,q}F(\R)$ ($1\leq q\leq \infty$) in the following figure.

$$
\begin{tikzpicture}

\draw[->, ](0,0) -- (8,0);
\draw[->, ] (0,-2) -- (0,2);
\draw[->, ] (7.7,2) -- (5.7,1.7);
\node [] at (8.5,2.3) {{\rm critical line}};

\node[below] at (-0.2,0) {$0$};
\node  at (4,-0.3) {1};
\node [right] at (3.8,2.0) {$t=\frac{1}{p}-1$};

\node [left] at (8,-0.5) {$\frac{1}{p}$};
\node [left] at (0,2.0) {$t$};

\node [] at (2.3,1.1) {$ F^{td}_{p,q}(\R)  \hookrightarrow S^{t}_{p,q}F(\R)$};

\node [] at (4,-1.3) {{  $ S^{t}_{p,q}F(\R) \hookrightarrow F^{td}_{p,q}(\R)$}};
\node [] at (6.5,0.7) {{\rm not comparable}};

\fill (0,0) circle (2pt);
\fill (4,0) circle (2pt);
\path[draw, line width=1.5pt](0,0) -- (4,0);
\path[draw, line width=1.5pt](4,0) -- (5.8,2);
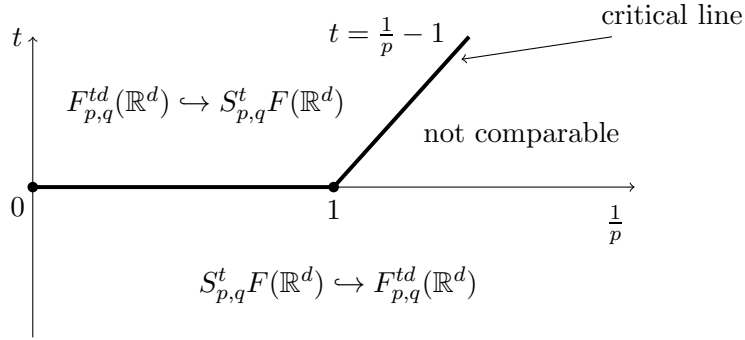
\node [right] at (-0.4,-2.8) {Figure 2. Comparison of $S^{t}_{p,q}F(\R)$ and $F^{td}_{p,q}(\R)$};
\end{tikzpicture}
$$
These embeddings are optimal in the following sense.

\begin{theorem}\label{tri-li2}
 Let $d\geq 2$, $0<p_0,p<\infty$, $0< q_0,q \le  \infty$ and $t_0,t\in \re$. 
Let $p,q$ and $t$ be fixed. 
 Within all spaces   $F^{t_0}_{p_0,q_0} (\R)$ satisfying
\beqq
F^{t_0}_{p_0,q_0} (\R) \hookrightarrow S^{t}_{p,q} F(\R) 
\eeqq
the class $F^{td}_{p,q} (\R)$ is the largest.
\end{theorem}

\begin{theorem}\label{tri-li3}
 Let $d\geq 2$, $0<p_0,p<\infty$, $0< q_0,q \le  \infty$ and $t_0,t\in \re$. 
Let $p,q$ and $t$ be fixed. 
 Within all spaces   $S^{t_0}_{p_0,q_0}F (\R)$ satisfying
\beqq
F^{td}_{p,q} (\R) \hookrightarrow S^{t_0}_{p_0,q_0} F(\R) 
\eeqq
the class $S^{t}_{p,q} F(\R)$ is the smallest.
\end{theorem}


\section{Proofs of the main results}\label{sec:proof}


To prove our main results we will apply essentially four different tools:  
vector-valued Fourier multipliers;  complex interpolation; assertions on dual spaces and 
some  test functions. 
In what follows we collect what is needed.


\subsection{Tools from Fourier analysis}\label{sec:prel}


In this section we will collect the required tools from Fourier analysis.

For a locally integrable function $f:\R\to \C$ we denote by $Mf(x)$ the Hardy-Littlewood maximal function defined by
\begin{equation}
  (Mf)(x) = \sup\limits_{x\in Q} \frac{1}{|Q|}\int\limits_{Q}\,|f(y)|\,dy\quad,\quad x\in\R \label{maxfunc}
  \quad,
\end{equation}
where the supremum is taken over all cubes with sides parallel to the coordinate axes containing~$x$.
A vector-valued generalization of the classical Hardy-Littlewood maximal inequality is due to
Fefferman and Stein \cite{FeSt71}.
\begin{theorem}\label{feffstein}\rm For $1<p<\infty$ and $1 <q \leq \infty$ there exists a constant $C>0$, such that
$$\|Mf_{\bar{k}}|L_p(\ell_q)\| \leq C\|f_{\bar{k}}|L_p(\ell_q)\|$$
  holds for all sequences $\{f_{\bar{k}}(x)\}_{\bar{k}\in \N_0^d}$ of locally Lebesgue-integrable functions on $\R$.
\end{theorem}
\noindent We require a direction-wise version of (\ref{maxfunc})
\beqq
(M_if)(x)=\sup_{s>0}\frac{1}{2s}\int\limits_{x_i-s}^{x_i+s}|f(x_1,\cdots,x_{i-1},t,x_{i+1},\cdots,x_d)|dt,\ \quad i=1,...,d\,.
\eeqq
The following version of the Fefferman-Stein maximal
inequality is due to  Bagby \cite{Bag}, see also St\"ockert \cite{Stoe}.

\begin{theorem}\label{max0}\rm For $1<p<\infty$ and $1 <q \leq \infty$ there exists a constant $C>0$, such that for
any $i=1,...,d$
$$\|M_if_{\bar{k}}|L_p(\ell_q)\| \leq C\|f_{\bar{k}}|L_p(\ell_q)\|$$
  holds for all sequences $\{f_{\bar{k}}(x)\}_{\bar{k}\in \N_0^d}$ of locally Lebesgue-integrable functions on $\R$.
\end{theorem}

Iterative application  of this theorem yields a similar boundedness property for the operator 
$\mathcal{M}=M_d\circ\cdots\circ M_1 $.

The following proposition will be a consequence of Theorem \ref{max0}. 
In it's proof we will  follow  the arguments in the isotropic case, see \cite{Ya}\,.

\begin{proposition}\label{four-1}
Suppose $1<p<\infty$, $1\leq q\leq \infty$ and let $\phi(x)\in \cs(\R)$. Then there exists a constant $C$ such that 
\beqq
\big\|\gf^{-1}[\phi(2^{-\bar{k}}\xi)\gf f_{\bar{k}}(\xi)](\cdot) \big|L_p(\ell_q)\big\|\leq C  \| f_{\bar{k}}(x)\big|L_p(\ell_q) \|
\eeqq
for all $\{f_{\bar{k}} \}_{\bar{k}\in \N_0^d}\in L_p(\ell_q)$.
\end{proposition}

\begin{proof}{\it Step 1.} 
The case $1<q\leq \infty$. Recall the notation $2^{\bar{k}} :=(2^{k_1},...,2^{k_d})$.
Observe that for $\bar{k} \in \N_0^d$ we have
\be
\label{equa}
\gf^{-1}[\phi(2^{-\bar{k}}\diamond\cdot)\gf f_{\bar{k}}(\cdot)](x) 
= (2\pi)^{-\frac{d}{2}}\, 2^{|\bar{k}|_1}\, \int\limits_{\R} (\gf^{-1}\phi)(2^{\bar{k}}\diamond y)f_{\bar{k}}(x-y)d y\,.
\ee
 Let $\alpha>1$. The assumption $ \phi \in \cs(\R)$ implies
\be \label{est-1}
\begin{split}
\bigg|& \int\limits_{\R} (\gf^{-1}\phi)(2^{\bar{k}}\diamond y)f_{\bar{k}}(x-y) d y \bigg|\\
&\leq \sup_{y\in \R}\bigg\{\Big(\prod_{i=1}^{d}(1+|2^{k_i}y_i|^2)^{\frac{\alpha}{2}}
\Big)|(\gf^{-1}\phi)(2^{\bar{k}}\diamond y)|\bigg\} \int\limits_{\R} \Big(\prod_{i=1}^{d}(1+|2^{k_i}y_i|^2)^{-\frac{\alpha}{2}}
\Big)|f_{\bar{k}}(x-y)| d y\,\\
& \leq c_1 \, \int\limits_{\R} \Big(\prod_{i=1}^{d}(1+|2^{k_i}y_i|^2)^{-\frac{\alpha}{2}}\Big)|f_{\bar{k}}(x-y)| d y
\end{split}
\ee
with a constant $c_1$ depending on $\phi$, but not on $\bar{k}$ and $f_{\bar{k}}$.
For $\bar{\ell}\in \Z$ and $\bar{k} \in \N_0^d$ we put 
\[
P(\bar{k},\bar{\ell}):=\big\{x\in \R:~~2^{-k_i}\, 2^{\ell_i}\le |x_i|<  2^{-k_i}\, 2^{\ell_i+1}\, ,\ \ i=1,...,d\big\}.
\]
Observe
\[
\R = \bigcup_{\bar{\ell}\in \Z} P(\bar{k},\bar{\ell}) \qquad \mbox{for all}\quad \bar{k} \in \N_0^d\, .
\]
Then we obtain from \eqref{est-1}
\be\label{app-max}
\begin{split} 
\bigg|\int\limits_{\R} (\gf^{-1}\phi)(2^{\bar{k}}\diamond y)&f_{\bar{k}}(x-y)d y \bigg|\\
& \leq c_1\,  \sum_{\bar{\ell}\in \zz ^d}\bigg(\sup_{y\in P(\bar{k},\bar{\ell})} 
\prod_{i=1}^{d}(1+|2^{k_i}y_i|^2)^{-\frac{\alpha}{2}}\bigg)\int\limits_{ P(\bar{k},\bar{\ell})} |f_{\bar{k}}(x-y)| d y.
\end{split}
\ee
By applying $\mathcal{M} $ to the integral on the right-hand side of \eqref{app-max} we derive
\beqq
\begin{split}
\bigg|\int\limits_{\R} (\gf^{-1}\phi)(2^{\bar{k}}\diamond y)f_{\bar{k}}(x-y)d y \bigg| 
&\leq c_1 (\mathcal{M} f_{\bar{k}})(x)\sum_{\bar{\ell}\in \zz ^d} 2^{-|\bar{k}|_1}  
\sup_{y\in P(\bar{k},\bar{\ell})} \prod_{i=1}^{d}\frac{2^{\ell_i}}{(1+|2^{k_i}y_i|^2)^{ \frac{\alpha}{2}}}\\
&\leq c_1\, 2^{-|\bar{k}|_1}(\mathcal{M} f_{\bar{k}})(x)\sum_{\bar{\ell}\in \zz ^d}   
\prod_{i=1}^d \frac{2^{\ell_i}}{(1+2^{\ell_i})^{\alpha}}\\
&\leq c_2\, 2^{-|\bar{k}|_1}(\mathcal{M} f_{\bar{k}})(x)\,.
\end{split}
\eeqq
Inserting this into \eqref{equa} we arrive at 
\beqq
|\gf^{-1}[\phi(2^{-\bar{k}}\diamond\cdot)\gf f_{\bar{k}}(\cdot)](x)| \leq c_3 \, (\mathcal{M} f_{\bar{k}})(x)
\eeqq
with some $c_3$ independent of $\bar{k}$ and $f_{\bar{k}}$.
Now the desired estimate follows from Theorem \ref{max0}.
\\
{\it Step 2.} The case $q=1$. From Step 1 we derive that the linear operator
\beqq
T:  \big\{f_{\bar{k}}\big\}_{\bar{k}} \to \big\{\gf^{-1}[\phi(2^{-\bar{k}}\diamond y)
\gf f_{\bar{k}}( y)]\big\}_{\bar{k}}
\eeqq
is bounded from $L_{p'}(\ell_{\infty})$ into itself. By a duality argument we  
conclude that the adjoint operator $T'$ of $T$ is bounded from $[L_{p'}(\ell_{\infty})]'$ into itself. That is
\beqq
\big\|\gf^{-1}[\overline{\phi}(2^{-\bar{k}}\diamond y)\gf f_{\bar{k}}( y)](\cdot) \,
\big|\,[L_{p'}(\ell_{\infty})]'\big\|\leq C  \|\, f_{\bar{k}} \,|\,[L_{p'}(\ell_{\infty})]' \|
\eeqq
for all $\{f_{\bar{k}}\}_{\bar{k} \in \N_0^d}\in [L_{p'}(\ell_{\infty})]'$. 
Of course the same inequality follows with $\overline{\phi}$ replaced by $\phi$.
The canonical embedding 
\[
L_p(\ell_1) \to  \Big[L_{p'}(\ell_\infty) \Big]' 
\]
is a linear isometry, see, e.g., \cite[Satz~III.3.1]{We}. 
We put 
\beqq
\begin{split}
\ca :=\Big \{\{f_{\bar{k}}(x)\}_{\bar{k} \in \N_0^d}: ~~ & f_{\bar{k}}(x) \in \cs(\R) \, , \quad  \\
& f_{\bar{k}}(x)\equiv 0\ \ \text{for all but a finite number of } \ \bar{k}\Big\}\,.
\end{split}
\eeqq
It is obvious that 
$$
\big\{\gf^{-1}[\phi(2^{-\bar{k}}\diamond y)\gf f_{\bar{k}}( y)]\big\}_{\bar{k} 
\in \N_0^d}\in \ca \subset L_p(\ell_1)
$$ 
if $\{f_{\bar{k}}(x)\}_{\bar{k} \in \N_0^d} \in \ca$. 
Because $\ca$ is dense in $L_p(\ell_1)$ we conclude that 
\beqq
\big\|\gf^{-1}[\phi(2^{-\bar{k}}\diamond y)\gf f_{\bar{k}}( y)](\cdot) \big|L_p(\ell_1)\big\|\leq C  \, 
\| f_{\bar{k}} \big|L_p(\ell_1) \|
\eeqq
holds for  all $\{f_{\bar{k}}\}_{\bar{k} \in \N_0^d}\in L_p(\ell_1)$. The proof is complete.
\end{proof}

\begin{definition}
Let $0<p,q\leq \infty$. Let $\Omega=\{\Omega_{\bar{k}}\}_{\bar{k}\in \N_0^d}$ be a sequence of 
compact subsets in $\R$. Then we define
\beqq
L_p^{\Omega}(\ell_q)=\Big\{\{f_{\bar{k}}\}_{\bar{k}\in \N_0^d}:\ f_{\bar{k}}\in \cs'(\R),\ 
\supp \gf f_{\bar{k}}\subset \Omega_{\bar{k}}\ \text{if }\ \bar{k}\in \N_0^d,\  \|f_{\bar{k}}\|L_p(\ell_q)\|<\infty\Big\}.
\eeqq
\end{definition}

We specify the sequence of compact subsets of $\R$ by choosing
\be\label{comp}
\Omega_{\bar{k}}:=\{x\in \R: \ |x_{k_i}|\leq a_{k_i},\ i=1,...,d\},
\ee
with $a_{\bar{k}}=(a_{k_1},...,a_{k_d})$, $\bar{k}\in \N_0^d$, $a_{k_i}>0,$ $\ i=1,...,d$.
The following Proposition has been proved in \cite[Theorem 1.10.2]{ST} for  $d=2$. 
A proof for  general $d$  can be found in \cite[Proposition 2.3.4]{Hansen}.

\begin{proposition}\label{max}
Let $0<p<\infty$, $0<q\leq \infty$ and $\Omega=\{\Omega_{\bar{k}}\}_{\bar{k}\in \N_0^d}$ be the 
sequence given in \eqref{comp}. Let $0<s<\min(p,q)$. Then there exists a positive constant $C$, independent of the sequence $\{a_{\bar{k}}\}_{\bar{k}}$, such that
\beqq
\bigg\| \sup_{z\in \R}\frac{|f_{\bar{k}}(\cdot-z)|}{\prod_{i=1}^{d}(1+|a_{k_i}z_i|^{\frac{1}{s}})} 
\bigg|L_p(\ell_q)\bigg\|\leq C\|f_{\bar{k}}|L_p(\ell_q)\|
\eeqq
holds for all systems $\{f_{\bar{k}}\} \in L_p^{\Omega}(\ell_q)$.
\end{proposition}

Next we recall a Fourier multiplier assertion for the spaces $L_p^{\Omega}( \ell_q) $. We refer to 
\cite[Theorem 1.10.3]{ST}, see also \cite[Theorem 1.12]{Vybiral} or \cite[Proposition 2.3.5]{Hansen}.

\begin{lemma}\label{mul1}
Let $0<p< \infty$, $0<q\leq\infty$ and $\Omega=\{\Omega_{\bar{k}}\}_{\bar{k}\in \N_0^d} $ be a 
sequence of compact subsets of $\R$ given in \eqref{comp}.  Let $r>\frac{1}{\min(p,q)}+\frac{1}{2}$. 
Then there exists a constant $C$, independent of the sequence $\{a_{\bar{k}}\}_{\bar{k}}$, such that 
\beqq
\| \gf^{-1}M_{\bar{k}}\gf f_{\bar{k}}\,|\,L_p( \ell_q)\| \,\leq\, C\sup_{\bar{\ell}\in \N_0^d} 
\| M_{\bar{\ell}}(a_{\bar{\ell}}\cdot)|S^r_2H(\R)\|\cdot \| f_{\bar{k}}|L_p( \ell_q)\|
\eeqq
holds for all systems $ \{f_{\bar{k}}\}_{\bar{k}} \in L_p^{\Omega}( \ell_q)$ and all systems $\{M_{\bar{k}}\}_{\bar{k}}\in S^r_2H(\R)$.
\end{lemma}


\subsection{Complex interpolation}
\label{cinterpol}


For the basics of Calder\'on's  complex interpolation method  we refer to 
the monographs \cite{BL,lun,t78}. It is well-known that this complex interpolation method can be extended 
to a special class  of quasi-Banach spaces, 
called analytically convex, see \cite{kmm}. Note that any Banach space is analytically convex. 
The following proposition was well-known in the classical context of   Banach spaces, see \cite[Theorem 2.1.6]{lun}, \cite[Theorem~4.1.2]{BL} 
or \cite[Theorem~1.10.3.1]{t78}. The extension to quasi-Banach spaces can be found in  Kalton, Mayboroda and Mitrea \cite{kmm}.

\begin{proposition}\label{inter1}
Let $0 < \Theta < 1$.
Let $(X_1,Y_1)$ and $(X_2,Y_2)$ be two compatible couples of quasi-Banach spaces. In addition, 
let $X_1+Y_1$, $X_2+Y_2$ be analytically convex. 
If $T$ is in $\mathscr{L}(X_1,X_2)$ and in $\mathscr{L}(Y_1,Y_2)$, then the restriction of $T$ to $[X_1,Y_1]_{\Theta}$ is in 
$\mathscr{L}([X_1,Y_1]_{\Theta},[X_2,Y_2]_{\Theta})$ for every $\Theta$. Moreover, 
\beqq
\|T: [X_1,Y_1]_{\Theta}\to [X_2,Y_2]_{\Theta}\| \,\leq\, \|T:X_1\to X_2 \|^{1-\Theta}\, \|T:Y_1\to Y_2\|^{\Theta}.
\eeqq
\end{proposition}

Complex interpolation of isotropic Lizorkin-Triebel spaces has been studied, e.g.,  in \cite{t78,Tr83,fj90,kmm}. 
For the case of dominating mixed smoothness one can 
find the proof for associated sequence spaces in \cite[Theorem 4.6]{Vybiral}. However, 
these results can be shifted to the level of function 
spaces by some wavelet isomorphisms, see \cite[Theorem 2.12]{Vybiral}.

\begin{proposition}\label{inter2}
Let $t_i\in \mathbb{R}$, $0<p_i<\infty$, $0<q_i\leq \infty$, $i=1,2$, and $\min(q_1,q_2)<\infty$. 
Let $0<\Theta< 1$. If $t_0,p_0$ and $q_0$ are given by
$$\frac{1}{p_0}=\frac{1-\Theta}{p_1}+\frac{\Theta}{p_2},\qquad \frac{1}{q_0}=\frac{1-\Theta}{q_1}+\frac{\Theta}{q_2},
\qquad t_0=(1-\Theta)t_1+\Theta t_2.$$
Then
\beqq
F^{t_0}_{p_0,q_0}(\R)=[F^{t_1}_{p_1,q_1}(\R),F^{t_2}_{p_2,q_2}(\R)]_{\Theta}
\eeqq
and
\beqq
S^{t_0}_{p_0,q_0}F(\R)=[S^{t_1}_{p_1,q_1}F(\R),S^{t_2}_{p_2,q_2}F(\R)]_{\Theta}.
\eeqq
\end{proposition}


\subsection{Dual spaces}\label{dualspaces}


Next we will recall some results  about the dual spaces of $F^t_{p,q}(\R)$ and $S^t_{p,q}F(\R)$. Note that  $\cs(\R)$ is dense either in 
$F^t_{p,q}(\R)$ or in  $S^t_{p,q}F(\R)$ if and only if  $\max(p,q)< \infty$.  By $\mathring{F}^t_{p,q}(\R)$ we denote the closure 
of $\cs(\R)$ in $F^t_{p,q}(\R)$ and by $\mathring{S}^t_{p,q}F(\R)$ the closure of $\cs(\R)$ in $S^t_{p,q}F(\R)$.  The dual space 
of $F^t_{p,q}(\R)$ must be understood in the following sense: $f\in \cs'(\R)$ belongs to the dual space $[F^t_{p,q}(\R)]'$   of 
$F^t_{p,q}(\R)$  if and only if there exists a positive constant $C$ such that
\beqq
|f(\varphi)|\leq C\|\varphi|F^t_{p,q}(\R)\|  \qquad \text{for all}\ \  \varphi \in \cs(\R).
\eeqq
Similarly for the space $S^t_{p,q}F(\R)$. For $1<p<\infty$ the conjugate exponent $p'$ is determined by $\frac{1}{p}+\frac{1}{p'}=1$.  
If $0<p\leq 1$ we put $p':=\infty$ and if $p=\infty$ we put $p':=1$. Let  $c_0$ be the space of all sequences converging to zero. Let $L_p(c_0)$ denote the space of all sequences $\{\psi_{\bar{k}}\}_{\bar{k}}$ of measurable functions such that
\[
\lim_{|\bar{k}|_1 \to \infty} | \psi_{\bar{k}}(x) | = 0 \ \qquad\text{a.e.}
\]
equipped with the norm
\[
 \| \, \psi_{\bar{k}}\, | L_p(c_0)\| :=   \Big\|\, \sup_{\bar{k}\in \N_0^d}|\psi_{\bar{k}}(\cdot)|\, \Big|L_p(\R) \Big\|\, .
\]
The following  lemma is well-known, see, e.g.,  \cite[Proposition 2.11.1]{Tr83} and \cite[Theorems 8.18.2, 8.20.3]{Ed}\,.

\begin{lemma}\label{lem:dual}
{\rm (i)} Let $1\leq p<\infty$ and $0<q<\infty$. Then $g\in (L_p(\ell_q))'$ if and only if it can be represented uniquely as 
\beqq
g(f)=\sum_{\bar{k}\in\N_0^d}\int\limits_{\R}g_{\bar{k}}(x)f_{\bar{k}} (x)dx
\eeqq
for every $f=\{f_{\bar{k}}(x)\}_{\bar{k}\in \N_0^d}\in L_p(\ell_q)$, where 
\beqq
 g=\{g_{\bar{k}}(x)\}_{\bar{k}\in \N_0^d}\in L_{p'}(\ell_{q'}) \quad \text{and}\quad \|g\|=\|g_{\bar{k}}| L_{p'}(\ell_{q'})\|\,.
\eeqq
{\rm (ii)} Let $1< p<\infty$. Then we have
\beqq
(L_p(c_0))' =L_{p'}(\ell_1)\,.
\eeqq 
\end{lemma}

\begin{proposition}\label{dual1}
Let $t\in \re$.\\
{\rm (i)} If $1< p <\infty$ and $1\leq q\leq \infty$, then
\beqq
[\mathring{F}^t_{p,q}(\R)]'=F^{-t}_{p',q'}(\R)\qquad \text{and}\qquad[\mathring{S}^t_{p,q}F(\R)]'=S^{-t}_{p',q'}F(\R).
\eeqq
{\rm (ii)} If $0<p<1$ and $0<q\leq\infty$ then 
\beqq
[\mathring{F}^t_{p,q}(\R)]'=B^{-t+d(\frac{1}{p}-1)}_{\infty,\infty}(\R)\qquad \text{and}\qquad[\mathring{S}^t_{p,q}F(\R)]'=S^{-t+\frac{1}{p}-1}_{\infty,\infty}B(\R).
\eeqq
\end{proposition}

\begin{proof}The proof in the isotropic setting can be found in \cite[Section 2.11]{Tr83} and \cite{Ma}. 
Duality of spaces of dominating mixed smoothness has been considered  in \cite[Section 5.5]{Hansen}. But there 
only partial results with respect to sequence spaces associated to Lizorkin-Triebel spaces of dominating mixed smoothness can be found. 
\\
{\it Step 1.} The case $1< p <\infty$ and $1\leq q<\infty$. 
\\
{\em Substep 1.1.}
We shall prove that $S^{-t}_{p',q'}F(\R)\hookrightarrow (S^t_{p,q}F(\R))'
$. Let $f\in S^{-t}_{p',q'}F(\R)$. Since $1<p'<\infty$ and $1<q'\leq \infty$ we can find $\{f_{\bar{k}} \}_{\bar{k}\in \N_0^d}$ such that 
\beqq
 f=\sum_{\bar{k}\in \N_0^d}\gf^{-1}\varphi_{\bar{k}}\gf f_{\bar{k}} \ \ \text{in} \ \cs'(\R) \quad \text{ and}
 \quad \|2^{-|\bar{k}|_1t}f_{\bar{k}}|L_{p'}(\ell_{q'})\| \leq 2\, \|f|S^{-t}_{p',q'}F(\R)\|^* \,.
 \eeqq
With  $\varrho \in \cs(\R)$ we conclude
\beqq
\begin{split}
|f(\varrho)|& = \Big|\sum_{\bar{k}\in \N_0^d}(\gf^{-1}\varphi_{\bar{k}}\gf f_{\bar{k}})(\varrho)\Big|  
=  \Big|\sum_{\bar{k}\in \N_0^d}   f_{\bar{k}}(\gf\varphi_{\bar{k}} \gf^{-1} \varrho) \Big|\\
 &\leq c_1 \,   \| 2^{-|\bar{k}|_1t}f_{\bar{k}}|L_{p'}(\ell_{q'})\|\cdot \| 2^{|\bar{k}|_1t}\gf \varphi_{\bar{k}} \gf^{-1}\varrho\, |L_p(\ell_q)\|\\
 &\leq c_2\, \|f|S^{-t}_{p',q'}F(\R)\|^* \, \cdot\,  \|\varrho |S^{t}_{p,q}F(\R)\|\,.
\end{split}
\eeqq
{\em Substep 1.2.}
Now we prove the reverse direction. 
Here we assume that the underlying decomposition of unity, see \eqref{varphi} and \eqref{decom},  is generated by an even function $\varphi_0$.
Then all elements of the sequence $\{\varphi_{\bar{k}}\}_{\bar{k} \in \N_0^d}$ are even functions as well.
Let $f\in S^t_{p,q}F(\R)$. Then the operator  
\beqq
f\to \{2^{|\bar{k}|_1t}\gf^{-1}\varphi_{\bar{k}}\gf f\}_{\bar{k}\in \N_0^d}
\eeqq
is one-to-one mapping from $S^t_{p,q}F(\R)$ onto a subspace $Y$ of $L_p(\ell_q)$. Hence, every functional $g\in (S^t_{p,q}F(\R))'$ 
can be interpreted as a functional on that subspace. 
From the Hahn-Banach theorem we derive that $g$ can be extended to a continuous linear functional on $L_p(\ell_q)$ with 
preservation of the norm. We still denote this extension by $g$. 
Now if $\varrho\in \cs(\R)$, then Lemma \ref{lem:dual} yields the existence of a sequence $\{g_{\bar{k}}\}_{\bar{k}}$ such that
\beqq
g(\varrho)=   \sum_{\bar{k}\in\N_0^d}\, \int\limits_{\R} g_{\bar{k}}(x) \, 2^{|\bar{k}|_1t} \, \gf^{-1}[\varphi_{\bar{k}}\gf \varrho](x)\, dx 
\eeqq
and 
\beqq
\| g\| =  \| g_{\bar{k}}|L_{p'}(\ell_{q'})\|\,.
\eeqq
Next we continue with a simple observations. 
Since $\varphi_{\bar{k}}$ is even we obtain
\[
 \varphi_{\bar{k}} (\xi) \, (\gf \varrho) (\xi) = \varphi_{\bar{k}} (-\xi) \, (\gf^{-1} \varrho) (-\xi)\, , \qquad \xi \in \R\, . 
\]
Applying the inverse Fourier transform to both sides of this identity
it follows
\[
\gf^{-1} [ \varphi_{\bar{k}} (\xi) \, (\gf \varrho) (\xi)](x) = \gf [\varphi_{\bar{k}} (\xi) \, (\gf^{-1} \varrho) (\xi)] (x) \, .
\]
We put 
$h_{\bar{k}}:=  2^{|\bar{k}|_1t}  \, g_{\bar{k}}$, $\bar{k} \in \N_0^d$, By $\langle \, \cdot \, , \, \cdot \, \rangle$ we denote the scalar 
product in $L_2 (\R)$. Because of  $\varphi_{\bar{k}}$ is real-valued, we find
\beqq
\int\limits_{\R}  \,  h_{\bar{k}}(x) \, \gf^{-1}[\varphi_{\bar{k}}\gf \varrho](x)  \, dx 
& = &   \langle  {h}_{\bar{k}}, ~ \overline{\gf^{-1}[\varphi_{\bar{k}}\gf \varrho]} \rangle
=  \langle {h}_{\bar{k}},   \, \gf [\varphi_{\bar{k}}  \gf^{-1} \overline{\varrho}] \rangle
 =    \langle {h}_{\bar{k}},   \, \gf^{-1} [\varphi_{\bar{k}}  \gf\overline{\varrho}] \rangle
\\
& = &  (\gf [\varphi_{\bar{k}}\gf^{-1} h_{\bar{k}}])(\varrho)\, .
\eeqq
Altogether this proves the identity
\[
g(\varrho)=   \sum_{\bar{k}\in\N_0^d} (\gf^{-1} [\varphi_{\bar{k}} \, \gf h_{\bar{k}}]) (\varrho)
\]
for all $\varrho \in \cs (\R)$. This means
\[
g =  \sum_{\bar{k}\in\N_0^d} \gf^{-1} [\varphi_{\bar{k}} \, \gf h_{\bar{k}}] \qquad \mbox{in}\quad \cs'(\R)
\]
and $\| g\| =  \|  2^{-|\bar{k}|_1t} \, h_{\bar{k}}|L_{p'}(\ell_{q'})\|< \infty$. 
In view of Proposition \ref{inf} we conclude that $g\in S^{-t}_{p',q'}F(\R)$.
\\
{\it Step 2.} The case $1 < p<\infty$ and $q=\infty$. We proceed as in Step 1. The mapping 
\beqq
f\to \{2^{|\bar{k}|_1t}\gf^{-1}\varphi_{\bar{k}}\gf f\}_{\bar{k}\in \N_0^d}
\eeqq
is an isometric mapping from $\mathring{S}^t_{p,\infty}F(\R)$ to $L_p(c_0)$. Here we use the fact that 
\beqq
\lim_{|\bar{k}|_1 \to \infty} \big|2^{|\bar{k}|_1t}  \gf^{-1}[\varphi_{\bar{k}}\gf f ](x) \big| = 0 \ \qquad\text{a.e.}
\eeqq
holds for all $f\in \mathring{S}^t_{p,\infty}F(\R)$. Now the assertion is obtained from Lemma \ref{lem:dual} (ii).
\\
{\it Step 3.} Proof of (ii). We have 
\beqq
\mathring{S}^{t}_{p,\min(p,q)}B(\R)\hookrightarrow \mathring{S}^{t}_{p,q}F(\R) \hookrightarrow 
\mathring{S}^{t-\frac{1}{p}+1}_{1,1}F(\R)=\mathring{S}^{t-\frac{1}{p}+1}_{1,1}B(\R)\,.
\eeqq
The known  duality relations  of the space $S^{t}_{p,q}B(\R)$, see \cite{KS} and references given there, yields
\beqq
S^{-t+ \frac{1}{p}-1}_{\infty,\infty}B(\R)\hookrightarrow (\mathring{S}^{t}_{p,q} F(\R))'\hookrightarrow 
S^{-t+ \frac{1}{p}-1}_{\infty,\infty}B(\R)\,.
\eeqq
This finishes the proof.
\end{proof}

 
\subsection{Test function}\label{test}


Let us give some few more properties of our smooth decompositions of unity in Section \ref{sec:def}. 
As a consequence of the definitions we obtain for $j\in \N$
\[
\varphi_{j} (\xi ) = 1 \qquad \mbox{if}\qquad \frac 34 \, 2^{ j} \le \xi \le 2^{ j}\, ,    
\]
 and  $\ell \in \N$ 
\[
\psi_\ell (x) = 1 \quad \mbox{on the set} \quad \Big\{x:\ \sup_{j=1,...,d} |x_j| \le 2^{\ell}\Big\}\setminus 
\Big\{x: \ \sup_{j=1,...,d} |x_j| \le \frac 34 \, 2^{\ell}\Big\}\, .
\]


\subsubsection*{Example 1}


Let us fix $\ell\in \N$ and let $\eta \in \cs(\re)$ such that 
$
\supp \gf \eta\subset \{\xi\in \re: 0<\xi< \frac{1}{4}\}.
$ 
We define the function $g_{\ell}$ by its Fourier transform
\beqq
\gf g_{\ell}(\xi)=\sum_{j=1}^{\ell}a_j(\gf \eta)(\xi-\frac{7}{8}2^j).
\eeqq
Then we arrive at
\beqq
\gf^{-1}\big(\varphi_j\gf g_{\ell}\big)(\xi)=a_je^{\frac{7}{8}2^{j}i\xi}\eta(\xi)\, , \qquad j \le \ell\, .
\eeqq
Consequently we obtain
\beqq
  \|g_{\ell}|F^0_{p,q}(\re) \|= \| \, \eta\, |L_p (\R)\| \, \Big(\sum_{j=1}^{\ell} |a_j|^q\Big)^{\frac{1}{q}}\,.
\eeqq
Now we turn to the multi-dimensional case and introduce a new family of test functions  $f_{\ell}: \R\to \C$ as follows:
\beqq
\gf f_\ell(x)=\theta_{\ell}(x_1)\cdots \theta_{\ell}(x_{d-1})(\gf g_{\ell})(x_d)\,,\qquad x=(x_1,...,x_d)\in \R\,,
\eeqq
where $\theta_1\in \cs(\re)$ is a function satisfying
\beqq
\supp \theta_1 \subset \{\xi:\ \varphi_1(\xi)=1\}\qquad \text{and}\qquad \theta_{\ell}(\xi)=\theta_1(2^{-\ell+1}\xi).
\eeqq
Clearly, 
\beqq
\supp\theta_{\ell}\subset \{\xi:\ \varphi_{\ell}(\xi)=1 \}\qquad \text{and}\qquad \supp (\gf f)\subset \{x:\ \psi_{\ell}(x)=1\}.
\eeqq 
By means of the cross norm property 
we obtain
\be \label{ws-10}
\begin{split} 
\|f_{\ell}|S^0_{p,q}F(\R)\|     
&\, = \, \Big(\prod_{j=1}^{d-1} \| \cfi \theta_{\ell} \, |F^0_{p,q}(\re)\big\|\Big) \, \| g_{\ell}|F^0_{p,q}(\re)\big\| 
\\
 & =  \, 
\Big(\prod_{j=1}^{d-1} \| \cfi \theta_{\ell} \, |L_{p}(\re)\big\|\Big) \, \| g_{\ell}|L_{p}(\re)\big\| 
\\ 
&=  \,
 2^{(\ell-1)(1-\frac{1}{p})(d-1)} \, \| \, \cfi \theta_1\, |L_p (\R)\|^{d-1} \, \| g_{\ell}|F^0_{p,q}(\re)\big\|  
\\
& = \,  C_1  \,  2^{\ell(1-\frac{1}{p})(d-1)}\Big(\sum_{j=1}^{\ell}|a_j|^q\Big)^{\frac{1}{q}}
\end{split}
\ee
and
\beqq
 \|f_{\ell}|F^0_{p,q}(\R)\|    & = & 
\Big(\prod_{j=1}^{d-1} \| \cfi \theta_{\ell} \, |L_{p}(\re)\big\|\Big) \, \| g_{\ell}|L_{p}(\re)\big\| 
\\
& = &
C_1 \,2^{\ell(1-\frac{1}{p})(d-1)}\| g_{\ell}|L_p(\re)\big\| 
\eeqq
for appropriate positive constant $C_1$. Using the Littlewood-Paley characterization of $L_p (\re)$, $1<p<\infty$, it follows
\begin{equation}\label{ws-09}
\begin{split}  \|f_{\ell}|F^0_{p,q}(\R)\|  \asymp   2^{\ell(1-\frac{1}{p})(d-1)}\Big(\sum_{j=1}^{\ell}|a_j|^2\Big)^{\frac{1}{2}}\,, 
\qquad \ell \in \N\, .
\end{split}
\end{equation}


\subsubsection*{Example 2}


Let us consider a function $g\in C_0^{\infty}(\re)$ such that 
$\supp g\subset \{t\in \re: 3/4 \leq |t| \leq 1 \}$. For  $\ell\in \N_0$ we define  
\beqq
g_{\ell}(t):= g(2^{-\ell}t) \qquad \text{and}\qquad
f_{\ell}(x) :=  \gf^{-1}\big[g_{\ell}(\xi_1)g_0(\xi_2)\cdots g_0(\xi_d)\big](x) \, .
\eeqq
Then we find
\beqq
  \|f_{\ell}|F^{t}_{p,q}(\R)\|  
   &= &    2^{t\ell} \big\|  \gf^{-1}\big[g_{\ell}(\xi_1)g_0(\xi_2)\cdot\, \ldots \, \cdot g_0(\xi_d)\big]\big|L_p(\R) \big\| 
\\
& = & \big\|  \gf^{-1}\big[g_{0}(\xi_1)g_0(\xi_2)\cdot\, \ldots \, \cdot g_0(\xi_d)\big]\big|L_p(\R) \big\|\, 2^{\ell(t+1-\frac{1}{p})}\,
 \eeqq 
and
\beqq
   \|f_{\ell}|S^{t}_{p,q}F(\R)\| & = &    2^{t\ell } \big\|  \gf^{-1}\big[g_{\ell}(\xi_1)g_0(\xi_2)\cdots g_0(\xi_d)\big]\big|L_p(\R) \big\| 
\\
& = & \big\|  \gf^{-1}\big[g_{0}(\xi_1)g_0(\xi_2)\cdot\, \ldots \, \cdot g_0(\xi_d)\big]\big|L_p(\R) \big\|\, 2^{\ell(t+1-\frac{1}{p})}\, .
 \eeqq


\subsubsection*{Example 3}


We consider the same  functions $g_{\ell}$ as in Example 2. This time we define   
\beqq
f_{\ell}(x) :=  \gf^{-1}\big[g_{\ell}(\xi_1)g_{\ell}(\xi_2)\cdots g_{\ell}(\xi_d)\big](x).
\eeqq
Let $G_\ell (\xi):= g_{\ell}(\xi_1)g_{\ell}(\xi_2)\cdot \ldots \cdot g_{\ell}(\xi_d)$.
As above we conclude
\beqq
  \|f_{\ell}|S^{t}_{p,q}F(\R)\| 
  =     2^{td\ell } \big\|  \gf^{-1} G_{{\ell} }\big|L_p(\R) \big\| = C_3\,  2^{d\ell(t+1-\frac{1}{p})}\, , \qquad \ell \in \N\, , 
 \eeqq 
and
\beqq
  \|f_{\ell}|F^{t}_{p,q}(\R)\|  
  =   2^{t\ell } \big\|  \gf^{-1} G_{{\ell}}\big|L_p(\R) \big\| = C_3\,  2^{d\ell(\frac{t}{d}+1-\frac{1}{p})}\, , \qquad \ell \in \N\, .
\eeqq


\subsubsection*{Example 4}


Let $g\in \cs(\R)$ with 
$\supp \gf g\subset  [0,\frac{1}{4}]^d.$
We define
\beqq
f_{\ell}(x):=  \sum_{j=1}^{\ell}a_j\, e^{i \frac{7}{8}2^jx_1}\, g(x) \,, \qquad x=(x_1,...,x_d)\in \R.
\eeqq
Then we have
$$
\gf f_{\ell}(\xi) = \sum_{j=1}^{\ell}\,  a_j\,  (\gf g)(\xi_1-\frac{7}{8}2^j,\xi_2,...,\xi_d)\,.
$$
We obtain
\[ 
\gf^{-1}[\varphi_{ \bar{k} }\gf f_{\ell}](x) = \sum_{j=1}^{\ell} \, \delta_{ \bar{k} ,(j,0,...,0)}\,  a_j \, e^{i \frac{7}{8}2^jx_1}\, g(x) 
\]
and
\[
 \gf^{-1}[\psi_{j}\gf f_{\ell}](x) =  a_j \, e^{i \frac{7}{8}2^jx_1}\, g(x)\,, \qquad j \le \ell\, .
\]
This leads to
\beqq
\|f_{\ell}|F^t_{p,q} (\R)\| = \|f_{\ell}|S^t_{p,q}F(\R)\|= \|g|L_p (\R)\| \, \Big(\sum_{j=1}^{\ell} 2^{jtq}|a_j|^q\Big)^{1/q}\,.
\eeqq


\subsubsection*{Example 5}


We shall  modify Example 4. This time we define the function
\beqq
f_{\ell}(x) := \sum_{j=1}^{\ell}\, a_j\, e^{i \frac{7}{8}2^j(x_1+...+x_d)}\, g(x) \,,\qquad x=(x_1,...,x_d)\in \R.
\eeqq
As above we conclude
\[
\gf^{-1}[\varphi_{\bar{k}}\gf f_{\ell}](x) = \sum_{j=1}^\ell 
\delta_{ \bar{k} , (j, \ldots ,j)} \, a_j \, e^{i \frac{7}{8}2^j(x_1+...+x_d)}\, g(x) 
\] 
and
\[  
\gf^{-1}[\psi_{j}\gf f_{\ell}](x) =  a_j\, e^{i \frac{7}{8}2^j(x_1+...+x_d)}\, g(x)\,, \qquad j \le \ell\, .
\]
Hence, we obtain
\beqq
\|f_{\ell}|F^t_{p,q} (\R)\| = \|g|L_p (\R)\| \, \Big(\sum_{j=1}^{\ell} 2^{jtq}|a_j|^q\Big)^{1/q}
\eeqq
and
\beqq
\|f_{\ell}|S^t_{p,q}F(\R)\| = \|g|L_p (\R)\| \, \Big(\sum_{j=1}^{\ell} 2^{djtq}|a_j|^q\Big)^{1/q}\,.
\eeqq


\subsubsection*{Example 6}


This example is taken from \cite[2.3.9]{Tr83}, see also \cite{KS}.
Let $\varrho \in \cs (\R) $ be a function such that 
$\supp \gf \varrho \subset \{\xi : \: |\xi|\le 1 \} $.
We define
\[
h_j(x):= \varrho (2^{-j} x)\, , \qquad x \in \R\, , \quad j \in \N\, .
\]
For all admissible $p,q,t$ we conclude
\beqq
\|\, h_j \, |S^t_{p,q}F(\R)\| = \|\, h_j \, |F^t_{p,q}(\R)\| = \|\, h_j \, |L_{p}(\R)\| = 
2^{jd/p}\, \|\, \varrho \, |L_{p}(\R)\|\, , \qquad j \in \N\, .   
\eeqq
{~}\\
As an immediate conclusion of this example we obtain 
the following result.

\begin{lemma}\label{p0p1}
Let $0 < p_0,p_1 <\infty$, $0< q_0,q_1 \le \infty$ and $t_0,t_1 \in \re$.
\\
{\rm (i)} An embedding $S^{t_0}_{p_0,q_0}F(\R) \hookrightarrow F^{t_1}_{p_1,q_1}(\R)$ implies $p_0 \le p_1$.
\\
{\rm (ii)} An embedding $F^{t_0}_{p_0,q_0}(\R) \hookrightarrow S^{t_1}_{p_1,q_1}F(\R)$ implies $p_0 \le p_1$.
\end{lemma}


\subsection{Proof of results in Section \ref{SF-F}}


{\bf Proof of Theorem \ref{main1}.} 
{\em Step 1.} Preparations.
For  $\bar{k}\in \N_0^d$ we define
\beqq
\square_{\bar{k}}: = \{ j\in \N_0 :\quad  \supp\psi_{j}\cap\supp\varphi_{\bar{k}} \not=\emptyset\}
\eeqq
and $j\in \N_0$ 
\beqq
\Delta_{j}: = \{ \bar{k}\in \N_0^d :\quad  \supp\psi_{j}\cap\supp\varphi_{\bar{k}} \not=\emptyset\}.
\eeqq
The condition  
$  \supp\psi_{j}\cap\supp\varphi_{\bar{k}} \not=\emptyset $ 
implies  
\be \label{ws-02}
\max_{i=1, \ldots \, , d}\,  k_i-1\leq j\leq \max_{i=1,\ldots\, , d}\,  k_i+1.
\ee
Consequently we obtain
\beqq
 |\square_{\bar{k}}| \asymp 1 \, , \ \ \bar{k} \in \N_0^d\, 
\qquad\text{and}\qquad
  |\Delta_{j}| \asymp (1+j)^{d-1}\, , \ \ j \in \N_0\, .  
\eeqq
By definition we have
\be\label{ct1b}
\psi_j(x)=\sum_{\bar{k}\in \Delta_{j}}\varphi_{\bar{k}}(x)\psi_j(x) \, , \qquad x \in \R\, .
\ee
{\it Step 2.} The case $t>0$. From \eqref{ct1b} we have
\be\label{ct6}
\|f|F^{t}_{p,q}(\R)\| 
\, =\, \bigg\|\bigg( \sum_{j=0}^{\infty} \Big|\sum_{ \bar{k} \in \Delta_{j}}2^{tj-t|\bar{k}|_1}2^{t|\bar{k}|_1}
\gf ^{-1}[\varphi_{\bar{k}} \psi_{j}\gf f]\, \Big|^q\bigg)^{1/q}\bigg|L_p(\R)\bigg\|.
\ee
If $q\leq 1$, then
\be\label{ct8-1}
\begin{split}
\|f|F^{t}_{p,q}(\R)\|
&\leq \bigg\|\bigg( \sum_{j=0}^{\infty} \sum_{ \bar{k} \in \Delta_{j}}\big|2^{t(j-|\bar{k}|_1)}2^{t|\bar{k}|_1}
\gf ^{-1}[\varphi_{\bar{k}} \psi_{j}\gf f]\,\big|^q\bigg)^{1/q}\bigg|L_p(\R)\bigg\|\\
&\leq \,c_1  \bigg\|\bigg( \sum_{j=0}^{\infty} \sum_{ \bar{k} \in \Delta_{j}}\big|2^{t|\bar{k}|_1}\gf ^{-1}
[\varphi_{\bar{k}} \psi_{j}\gf f]\, \big|^q\bigg)^{1/q}\bigg|L_p(\R)\bigg\|\, .
\end{split}
\ee
The last inequality is due to 
\[
\sup_{j \ge 0}\, \sup_{\bar{k} \in \Delta_{j}}\, 2^{t(j-| \bar{k} |_1)}\leq c_1 < \infty \, .
\] 
Now we turn to $q>1$. Using H\"older's inequality  we obtain from \eqref{ct6}
\beqq 
\bigg|\sum_{ \bar{k} \in \Delta_{j}}2^{tj-t|\bar{k}|_1}2^{t|\bar{k}|_1}\gf ^{-1}[\varphi_{\bar{k}} \psi_{j}\gf f] \, \bigg|
\leq \bigg(\sum_{ \bar{k} \in \Delta_{j}} |2^{t|\bar{k}|_1} \, \gf ^{-1}[\varphi_{\bar{k}} \psi_{j}\gf f]\, \big|^q\bigg) ^{{1}/{q}} 
\bigg( \sum_{ \bar{k} \in \Delta_{j}}2^{t(j-|\bar{k}|_1)q'}\bigg)^{{1}/{q'}},
\eeqq
where $\frac{1}{q}+\frac{1}{q'}=1$. 
Because of $t>0$ and \eqref{ws-02} the second sum on the right-hand side is uniformly bounded, i.e., 
\[
\sup_{j \ge 0}\, 
\bigg( \sum_{ \bar{k} \in \Delta_{j}}2^{t(j-|\bar{k}|_1)q'}\bigg)^{{1}/{q'}} \le c_2 < \infty \, .
\]
 Consequently, we obtain for all $q$
\beqq
\|f|F^{t}_{p,q}(\R)\|
\leq \, c_3\, \bigg\|\bigg( \sum_{j=0}^{\infty} 
\sum_{ \bar{k} \in \Delta_{j}}\big|2^{t|\bar{k}|_1}\, \gf ^{-1}[\varphi_{\bar{k}} \psi_{j}\gf f]\,\big|^q\bigg)^{1/q}\bigg|L_p(\R)\bigg\|\,  .
\eeqq
Let $\tau := \min (1,p,q)$. Interchanging the order of summation we find
\beqq
\begin{split}
\|f|F^{t}_{p,q}(\R)\|^{\tau}
&\leq c_3^{\tau}\, \bigg\|\bigg( \sum_{ \bar{k} \in \N_0^d} \sum_{j\in \square_{\bar{k}}}
\big|2^{t|\bar{k}|_1}\gf ^{-1}[\varphi_{\bar{k}} \psi_{j}\gf f]\, \big|^q\bigg)^{1/q}\bigg|L_p(\R)\bigg\|^{\tau}
\\
&\leq c_3^{\tau}\,  \sum_{i=-1}^1 \, \bigg\|\bigg( \sum_{ \bar{k} \in \N_0^d} \big|2^{t|\bar{k}|_1}\, 
\gf ^{-1}[\varphi_{\bar{k}} \psi_{j+i}\gf f]\, \big|^q\bigg)^{1/q}\bigg|L_p(\R)\bigg\|^{\tau},
\end{split}
\eeqq
where $j:=|\bar{k}|_{\infty}$, see \eqref{ws-02}. 
We estimate the term with $i=0$.  
The terms with $i= \pm 1$ can be treated in a similar way.
Let $\{\tilde{\varphi}_{\bar{k}}\}_{ \bar{k} \in \N_0^d}$ be the system defined in the proof of Proposition \ref{inf}. Then we have
\beqq
\begin{split}
\bigg\|\bigg( \sum_{ \bar{k} \in \N_0^d} \big|2^{t|\bar{k}|_1}\, \gf ^{-1}[\varphi_{\bar{k}} \psi_{j}&\gf f] \, \big|^q\bigg)^{1/q} 
\bigg|L_p(\R)\bigg\| \\
&  = 
 \, \bigg\|\bigg( \sum_{ \bar{k} \in \N_0^d} \big|\gf ^{-1} \tilde{\varphi}_{\bar{k}}\psi_{j}\gf \big[2^{t|\bar{k}|_1}\gf ^{-1} 
\varphi_{\bar{k}} \gf f\big]\big|^q\bigg)^{1/q}\bigg|L_p(\R)\bigg\|.
\end{split}
\eeqq
Applying Lemma \ref{mul1} with $M_k=\tilde{\varphi}_{\bar{k}} \psi_{j} $ we obtain
\be
\begin{split} \label{ct9-1}
\bigg\|&\bigg( \sum_{ \bar{k} \in \N_0^d}  \big|2^{t|\bar{k}|_1}\gf ^{-1}\varphi_{\bar{k}} \psi_{j} \gf f\big|^q\bigg)^{1/q} \bigg|L_p(\R)\bigg\| \\
&\, \leq c_4\, \sup_{ \bar{k} \in \N_0^d} 
\| (\tilde{\varphi}_{\bar{k}}\psi_j)(2^{ \bar{k}+\bar{1} }\diamond\cdot)|S^r_{2}W(\R)  \|  
\bigg\|\bigg( \sum_{ \bar{k} \in \N_0^d}\big| 2^{t|\bar{k}|_1}\gf ^{-1} \varphi_{\bar{k}} \gf f \big|^q\bigg)^{1/q}\bigg|L_p(\R)\bigg\| 
\end{split}
\ee
where we have chosen  $r\in \N$ such that  $r>\frac{1}{\min(p,q)}+\frac{1}{2}$. 
To estimate the factor $\|\, \ldots \, |S^r_{2}W(\R)\| $ we consider several cases. First, we assume that $\min_{i=1, \ldots \, d} k_i \ge 1$. Then it follows
\[
(\tilde{\varphi}_{\bar{k}} \psi_j)(2^{ \bar{k}+\bar{1} }\diamond x) = 
\tilde{\varphi}_{\bar{1}} (4x) \,  \psi_1 (2^{k_1-j+2}x_1, \ldots \, , 2^{k_d - j + 2}x_d)\, .
\]
For any $\alpha \in \N_0^d$, since $k_1-j+2 = k_1- |k|_\infty + 2 \le 2 $, we conclude the existence of a positive constant $C_\alpha$ such that 
\beqq
\sup_{x\in \R}|D^{\alpha} (\psi_{1}(4 \, 2^{ \bar{k} -|k|_\infty}\diamond x))|  \leq  C_{\alpha} < \infty\, .
\eeqq
Furthermore 
\beqq
 \| \tilde{\varphi}_{\bar{k}}(2^{ \bar{k} +\bar{1} }\diamond\cdot)\, |S^r_{2}W(\R)  \|  =  \| \tilde{\varphi}_{\bar{1}}(4\cdot)\, | S^r_{2}W(\R)   \| = C_r < \infty\, , 
\eeqq 
which implies 
\[
\sup_{ \bar{k} \in \N^d } \| (\tilde{\varphi}_{\bar{k}}\psi_j)(2^{ \bar{k} +\bar{1} }\diamond\cdot)|S^r_{2}W(\R)  \|  \leq c_5 \,.
\]
Now we turn to the cases $\min_{i=1, \ldots \, , d} k_i =0$. Let us assume that $0=k_1 = \ldots =k_m < k_{m+1}\le \ldots \le k_d$ for some $m < d$.
Recall the notation from \eqref{eins}. Then 
\beqq
(\tilde{\varphi}_{\bar{k}} \psi_j)(2^{ \bar{k}+\bar{1} }\diamond x)  & = & 
\tilde{\varphi}_{0} (2x_1)\, \cdot \ldots \cdot  \, \tilde{\varphi}_{0} (2x_m) \, \cdot\,  \tilde{\varphi}_{1} (4x_{m+1})\cdot \ldots \cdot
\tilde{\varphi}_{1} (4x_{d})
\\
&& \times \quad 
\psi_1 (2^{-j+2}x_1, \ldots \, , 2^{- j + 2}x_m, 2^{k_{m+1} - j + 2}x_d,\ldots \, , 2^{k_d - j + 2}x_d)\, .
\eeqq
Now we proceed as above and find also in this case
\[
\| (\tilde{\varphi}_{\bar{k}}\psi_j)(2^{ \bar{k} +\bar{1} }\diamond\cdot)|S^r_{2}W(\R)  \|  \leq c_6 \,. 
\]
Similarly we can treat all cases caused by a different ordering of the components of $\bar{k}$.
The case $\bar{k} = \bar{0}$, $j=0$ can be handled in the same way.
Summarizing, we get
\be\label{ws-200}
\sup_{ \bar{k} \in \N^d_0 } \| (\tilde{\varphi}_{\bar{k}}\psi_j)(2^{ \bar{k} +\bar{1} }\diamond\cdot)|S^r_{2}W(\R)  \|  \leq c_7 < \infty \,.
\ee
From \eqref{ws-200}  and \eqref{ct9-1} we derive 
\beqq
\bigg\|\bigg( \sum_{ \bar{k} \in \N_0^d}  \big|2^{t|\bar{k}|_1}\gf ^{-1}\varphi_{\bar{k}} \psi_{j} \gf f\big|^q\bigg)^{1/q} \bigg|L_p(\R)\bigg\|  \leq c_8\, 
  \|f|S^t_{p,q}F(\R)\| \, .
\eeqq
We conclude that $S^t_{p,q}F(\R)\hookrightarrow F^t_{p,q}(\R)$.
\\
{\it Step 3.}  The case $t=0$. If $0<q\leq 1$ we can argue as in \eqref{ct8-1} with $t=0$. This implies 
\beqq
\|f|F^{0}_{p,q}(\R)\|
\leq  \Big\|\Big( \sum_{j=0}^{\infty} \sum_{ \bar{k} \in \Delta_{j}}\big|\gf ^{-1}[\varphi_{\bar{k}} \psi_{j}\gf f]\, \big|^q\Big)^{1/q}\Big|L_p(\R)\Big\|.
\eeqq
Then we continue as in Step 2 resulting in  $S^0_{p,q}F(\R)\hookrightarrow F^0_{p,q}(\R)$ if $0 < q \le 1$.
\\
Now we consider the case $1<q<2$. Here we employ complex interpolation. 
For $0<p<\infty$ and $1<q<2$ there exist  $\Theta\in(0,1)$, $0<p_0<\infty$ and $1< p_1<\infty$ such that 
$$\frac{1}{p}=\frac{1-\Theta}{p_0}+\frac{\Theta}{p_1}\qquad \text{ and}\qquad \frac{1}{q}=\frac{1-\Theta}{1}+\frac{\Theta}{2}.$$ 
Proposition \ref{inter2} yields
\beqq
S^0_{p,q}F(\R)=[S^0_{p_0,1}F(\R),S^0_{p_1,2}F(\R)]_{\Theta}\quad\text{and}\quad F^0_{p,q}(\R)=[F^0_{p_0,1}(\R),F^0_{p_1,2}(\R)]_{\Theta}\,.
\eeqq 
Finally, the claim follows from Proposition \ref{inter1}. 
The proof is complete.
\qed

\begin{lemma} \label{q<2-f}
Let $d\geq 2$, $0<p<\infty$ and $0<q\leq \infty$. Then the embedding 
$$S^{0}_{p,q}F(\R)\hookrightarrow F^{0}_{p,q}(\R)\qquad \text{implies}\qquad q\leq 2.$$
\end{lemma}

\begin{proof}
{\em Step 1.} The case $1<p<\infty$. We use the test function from Example 1. 
The embedding $S^{0}_{p,q}F(\R)\hookrightarrow F^{0}_{p,q}(\R)$ implies
the existence of a constant $c>0$ such that 
\[
 2^{\ell(1-\frac{1}{p})(d-1)}\Big(\sum_{j=1}^{\ell} |a_j|^2\Big)^{1/2}    \le c \,    2^{\ell(1-\frac{1}{p})(d-1)} \Big(\sum_{j=1}^\ell   \, |a_j|^q  \Big)^{1/q}
\]
holds for all $\ell \in \N$ and all sequences $\{ a_j\}_j$,
see \eqref{ws-10} and \eqref{ws-09}.
This requires $q \le 2$. 
\\
{\em Step 2.} The case $0<p\leq 1$. Assume that $ S^0_{p,q}F(\R)\hookrightarrow F^0_{p,q}(\R)$ with $0<p\leq 1$ and $2<q\leq \infty$. 
Then we can find a triple $(p_1,q_1,\Theta)$ such that $\Theta\in (0,1)$, $ 1<p_1<2<q_1<\infty$,
\beqq
 \frac{1}{p_1}=\frac{\Theta}{p}+\frac{1-\Theta}{2}\qquad \text{and} \qquad \frac{1}{q_1}=\frac{\Theta}{q}+\frac{1-\Theta}{2}.
\eeqq
Complex interpolation, see  Propositions \ref{inter1} and \ref{inter2}, yields $S^0_{p_1,q_1}F(\R) \hookrightarrow F^0_{p_1,q_1}(\R)
$. But this is in contradiction with Step 1 since $ p_1>1$ and $q_1>2$.   
\end{proof}

\noindent
{\bf Proof of Theorem \ref{1<p,q<vc}.} As a consequence of  Theorem \ref{SF-F} and Lemma \ref{q<2-f} it will be enough to consider the case $t<0$. 
We assume $S^{t}_{p,q}F(\R)\hookrightarrow F^{t}_{p,q}(\R)$ if $t<0$. Observe that this implies 
$\mathring{S}^{t}_{p,q}F(\R)\hookrightarrow \mathring{F}^{t}_{p,q}(\R)$.
Then, by duality, see Proposition \ref{dual1}, we 
obtain $F^{-t}_{p',q'}(\R)\hookrightarrow S^{-t}_{p',q'}F(\R)$. Using the test function from 
Example 5 with $a_j:= \delta_{j,\ell}$ we can disprove this embedding.
\qed 
\\
\\
\noindent
{\bf Proof of Proposition \ref{pro:1}.} {\em Step 1.} Proof of (i).
Theorem \ref{main1} implies $\mathring{S}^{t}_{p,q}F(\R)\hookrightarrow \mathring{F}^{t}_{p,q}(\R)$ if $t>0$.
Proposition \ref{dual1} yields
$F^{-t}_{p',q'}(\R)\hookrightarrow S^{-t}_{p',q'}F(\R)$, if $1 <p<\infty$ and $1 \le q \le \infty$.
\\
{\em Step 2.} Proof of (ii).
Since $-t+\frac{1}{p}-1<-t+d(\frac{1}{p}-1)$ and $t< 0$ we can use $g_{\bar{k}}=e^{i \bar{k} x}$ as test 
functions to prove that $S^{-t+\frac{1}{p}-1}_{\infty,\infty}B(\R)$ and $B^{-t+d(\frac{1}{p}-1)}_{\infty,\infty}(\R)$ 
are not comparable.
Here one can use
\[
\|\, e^{i \bar{k} x} \, |B^{s}_{\infty,\infty}(\R)\|\asymp (1+|\bar{k}|_2)^s\, , \qquad k \in \N_0^d
\] 
and 
\[
\|\, e^{i \bar{k} x} \, |S^{s}_{\infty,\infty}B(\R)\| \asymp \prod_{i=1}^d (1+|{k_i}|)^s\, , \qquad k \in \N_0^d\, .
\] 
Then, from Proposition \ref{dual1},  we can conclude that $\mathring{S}^t_{p,q}F(\R)$ and $\mathring{F}^t_{p,q}(\R)$ are incomparable
and therefore ${S}^t_{p,q}F(\R)$ and ${F}^t_{p,q}(\R)$ as well.
This finishes the proof.
\qed


\subsection{Proof of the results in Section  \ref{F-SF}} 


{\bf Proof of Theorem \ref{main2}.}
 The claim for the case $t=0$ is a consequence of Theorem \ref{main1} and a duality argument, see Proposition \ref{dual1}. 
The  proof in  case $t>\big(\frac{1}{\min(p,q)}-1\big)_+ $ will be  divided into several steps.
\\
{\it Step 1.}  We shall prove the embedding under the assumptions $0<p<\infty$, $0<q\leq \infty$ and $t>\frac{1}{\min(p,q)}$. 
Let $\tau := \min (1,p,q)$.
From Step 1 in the proof of Theorem \ref{SF-F} and  $\square_{\bar{k}}\asymp 1$ we obtain
 \be\label{ct1-1}
 \begin{split}
 \|f|S^{t}_{p,q}F(\R)\|^\tau  & = 
\bigg\|\bigg( \sum_{ \bar{k} \in \N_0^d} \Big|\sum_{j\in \square_{\bar{k}}}2^{ |\bar{k}|_1 t}\gf ^{-1}\varphi_{\bar{k}} \psi_{j}\gf f\Big|^q\bigg)^{1/q}
\bigg|L_p(\R)\bigg\|^\tau \\
  &\leq    \sum_{i=-1}^1 \bigg\|\bigg( \sum_{ \bar{k} \in \N_0^d}  \big|2^{ |\bar{k}|_1t}\gf ^{-1}\varphi_{\bar{k}} \psi_{j+i}\gf f\big|^q
\bigg)^{1/q}\bigg|L_p(\R)\bigg\|^\tau ,
   \end{split}
 \ee
 where again $j:=|\bar{k}|_{\infty}$. It will be enough to deal with  the term for  $i=0$. 
The others terms can be treated similarly.
 Since $t>\frac{1}{\min(p,q)}$ we can write $t=a+\varepsilon$ with $a>\frac{1}{\min(p,q)}$ and $\varepsilon>0$. We put
 \beqq
 g_{\bar{k}}(x) :=\gf ^{-1}\big[2^{(|\bar{k}|_1-jd)\varepsilon}2^{jtd} \psi_{j}\gf f\big](x)\, ,,  
\qquad \bar{k} \in \N_0^d,\quad j=|\bar{k}|_{\infty} \, .
 \eeqq
It follows
  \be \label{ct12}
 \bigg\|\bigg( \sum_{ \bar{k} \in \N_0^d}  \big|2^{ |\bar{k}|_1t}\gf ^{-1}[\varphi_{\bar{k}}  \psi_{j}\gf f]\, 
\big|^q\bigg)^{1/q}\bigg|L_p(\R)\bigg\|    
=   \bigg\|\bigg( \sum_{ \bar{k} \in \N_0^d} \big| 2^{(|\bar{k}|_1-jd)a}\gf ^{-1}[\varphi_{\bar{k}}\gf g_{\bar{k}} ]\, \big|^q
\bigg)^{1/q}\bigg|L_p(\R)\bigg\|\, .
\ee
Next we need the related Peetre maximal function. We define 
\[
P_{2^{\bar{j}+\bar{1}},a}g_{\bar{k}}(x):= \sup_{z\in \R}\, \frac{|g_{\bar{k}}(x-z)|}{\prod_{i=1}^{d} (1+|2^{j+1} z_i|^{a})}\, , 
\qquad x \in \R, \quad k \in \N_0^d\, , \quad  j=|\bar{k}|_{\infty} \, ,
\]
compare with Proposition \ref{max}.
A standard convolution argument, given by  
\beqq
\begin{split}
\big|\big(\gf ^{-1} [\varphi_{\bar{k}}&\gf g_{\bar{k}}]\big) (x-z)| 
\leq (2\pi)^{-d/2}   \int\limits_{\R} \big|(\gf ^{-1} \varphi_{\bar{k}}) (x-z-y)|\cdot  | g_{\bar{k}}(y) |d y \\
&\leq   (2\pi)^{-d/2}  P_{2^{\bar{j}+\bar{1}},a}g_{\bar{k}}(x) \int\limits_{\R} 
\big|(\gf ^{-1} \varphi_{\bar{k}}) (x-z-y)\big| \prod_{i=1}^{d}(1+|2^{j+1}(x_i-y_i)|^{a})\,  dy\, , 
\end{split}
\eeqq
the elementary  inequality 
\beqq
(1+|2^{j+1}(x_i-y_i)|^{a})\leq 2^a\,  (1+|2^{j+1}z_i|^{a})(1+|2^{j+1}(x_i-z_i-y_i)|^{a})\, , 
\eeqq
$i=1, \, \ldots \, ,d,$ and a  change of variable lead to
\[
\frac{ \big|\big(\gf ^{-1} \varphi_{\bar{k}}F g_{\bar{k}}\big) (x-z)\big|}{\prod_{i=1}^{d}(1+|2^{ j+1}z_i|^{a})} 
\leq c_1 \,   P_{2^{\bar{j}+\bar{1}},a}g_{\bar{k}}(x) 
\, \int\limits_{\R} \big|(\gf ^{-1} \varphi_{\bar{k}}) (y)\big| \prod_{i=1}^{d}(1+|2^{j+1}y_i|^{a}) d y\, .
\]
Temporarily we assume $\min_{i=1, \ldots \, , d} k_i \ge 1$. Then 
\[
 \int\limits_{\R} \big|(\gf ^{-1} \varphi_{\bar{k}}) (y)\big| \prod_{i=1}^{d}(1+|2^{j+1}y_i|^{a}) d y 
= \prod_{i=1}^d \int\limits_{\re} \big|\gf ^{-1} \varphi_{1} (t)\big|  (1+2^{j+2-k_i}|t|)^{a} d t
\]
follows. Since $k_i\leq j= |\bar{k}|_\infty$ and $\gf^{-1}\varphi_1\in \cs (\re)$ we have
\beqq
\begin{split}
 \int\limits_{\re} \big|\gf ^{-1} \varphi_{1} (t)\big|  (1+2^{j+2-k_i}|t|)^{a} d t & = 
\,2^{(j-k_i)a}     \int\limits_{\re} \big|\gf ^{-1} \varphi_{1} (t)\big|  (2^{k_i-j}+ 4|t|)^{a} d t 
\\
& \leq c_2\,  2^{(j-k_i)a} \, .
\end{split}
\eeqq
This estimate carries over to the situation $\min_{i=1, \ldots \, , d} k_i = 0$ by obvious modifications.
Consequently
\beqq
2^{(|\bar{k}|_1-jd)a}\, \frac{  \big|\big(\gf ^{-1}[\varphi_{\bar{k}}\gf g_{\bar{k}}]\big)(x-z)\big|}{\prod_{i=1}^{d}(1+2^{ j+1}|z_i|)^{a}}
 \leq c_3 \,  P_{2^{\bar{j}+\bar{1}},a}g_{\bar{k}}(x) 
\eeqq
with a constant $c_3$ independent of $x$ and $\{g_{\bar{k}}\}_{\bar{k}}$.
Obviously, this implies
\beqq
\begin{split}
2^{(|\bar{k}|_1-jd)a}\, \big|\big(\gf ^{-1} [\varphi_{\bar{k}}\gf g_{\bar{k}}]\big)(x)\big| 
&\leq  \sup_{z\in \R}\frac{ 2^{(|\bar{k}|_1-jd)a}\big|\big(\gf ^{-1}  
[\varphi_{\bar{k}}\gf g_{\bar{k}}]\big)(x-z)\big|}{\prod_{i=1}^{d}(1+|2^{ j+1}z_i|^{a})} \\
&  \leq c_3
 P_{2^{\bar{j}+\bar{1}},a}g_{\bar{k}}(x),
  \end{split}
\eeqq
which results in the estimate
  \beqq
  \begin{split} 
 \bigg\|\bigg( \sum_{ \bar{k} \in \N_0^d}  \big|2^{ |\bar{k}|_1t}\gf ^{-1} [\varphi_{\bar{k}}   \psi_{j}\gf f]\, \big|^q\bigg)^{1/q}
\bigg|L_p(\R)\bigg\|  \, \leq c_3 \,    
\bigg\|\bigg( \sum_{ \bar{k} \in \N_0^d} \big|P_{2^{\bar{j}+\bar{1}},a}g_{\bar{k}}(\cdot)\big|^q\bigg)^{1/q}\bigg|L_p(\R)\bigg\|,
        \end{split}
  \eeqq
see \eqref{ct12}. Now, applying Proposition \ref{max} with respect to  $\{g_{\bar{k}}\}_{ \bar{k} \in \N_0^d}$ and 
with $a_{k_i}$ chosen to be $2^{j+1}$, $i=1,...,d$, $j= |\bar{k}|_\infty$, we obtain
\beqq
\begin{split}
  \bigg\|\bigg( \sum_{ \bar{k} \in \N_0^d}   \big|2^{ |\bar{k}|_1t}\gf ^{-1}[\varphi_{\bar{k}}   \psi_{j}\gf f] 
&\big|^q\bigg)^{1/q}\bigg|L_p(\R)\bigg\|  
\\ 
&\leq c_4\,   \bigg\|\bigg( \sum_{ \bar{k} \in \N_0^d}  \big|\gf ^{-1} [2^{(|\bar{k}|_1-jd)\varepsilon}2^{jtd} \psi_{j}\gf f]\, 
\big|^q\bigg)^{1/q}\bigg|L_p(\R)\bigg\|
\\   
&\leq c_4 \,   \bigg\|\bigg( \sum_{ \bar{k} \in \N_0^d} \sum_{j\in \square_{\bar{k}}}
\big|\gf ^{-1} [2^{(|\bar{k}|_1-jd)\varepsilon}2^{jtd}   \psi_{j}\gf f]\, \big|^q\bigg)^{1/q}\bigg|L_p(\R)\bigg\|
\\ 
& =  c_4  \bigg\|\bigg( \sum_{j=0}^{\infty}  \big|2^{jtd}\gf ^{-1}[ \psi_{j}\gf f]\, \big|^q \sum_{ \bar{k} \in \Delta_{j}}2^{(|\bar{k}|_1-jd)
\varepsilon q}\bigg)^{1/q}\bigg|L_p(\R)\bigg\|.
\end{split}
\eeqq
Our assumption  $\varepsilon>0$ guarantees  
\beqq
\begin{split}
  \bigg\|\bigg( \sum_{ \bar{k} \in \N_0^d}   \big|2^{ |\bar{k}|_1t}\gf ^{-1} [\varphi_{\bar{k}}   \psi_{j}\gf f]\, \big|^q\bigg)^{1/q}
\bigg|L_p(\R)\bigg\| 
&\leq c_5\,   \bigg\|\bigg( \sum_{j=0}^{\infty}  \big|2^{jtd}\gf ^{-1} [\psi_{j}\gf f]\, \big|^q \bigg)^{1/q}\bigg|L_p(\R)\bigg\|\\
&= c_5\,  \|f|F^{td}_{p,q}(\R)\|\, ,
\end{split}
\eeqq
see \eqref{ws-02}. Inserting this into \eqref{ct1-1} and carrying out the estimates of the other terms in the same way, the claim follows.
\\
{\it Substep 2.} We shall prove the embedding under the assumptions 
 $1<p<\infty$, $1 \le q\leq \infty$ and $t>0$. This time we use Proposition \ref{four-1}.   
Starting point is inequality  \eqref{ct1-1}. As above it will be enough to deal with  $j:=|\bar{k}|_{\infty}$. 
The remaining  terms can be treated in a similar way.  
Applying Proposition \ref{four-1} in connection with a decomposition argument as in \eqref{zerl} we obtain
\beqq
\begin{split} 
 \bigg\|\bigg( \sum_{ \bar{k} \in \N_0^d}  \big|2^{ |\bar{k}|_1t}\gf ^{-1}[\varphi_{\bar{k}}& \psi_{j}\gf f]\, \big|^q\bigg)^{1/q} 
\bigg|L_p(\R)\bigg\| 
 \\
& = \, \bigg\|\bigg( \sum_{ \bar{k} \in \N_0^d} \big|\gf ^{-1} \varphi_{\bar{k}}\gf \big[ 2^{ |\bar{k}|_1 t}\gf ^{-1}   \psi_{j}\gf f
 \big]\big|^q\bigg)^{1/q}\bigg|L_p(\R)\bigg\|
\\ 
&\leq  \, c_6 \,  \bigg\|\bigg( \sum_{ \bar{k} \in \N_0^d} \big| 2^{ |\bar{k}|_1 t}\gf ^{-1} [\psi_{j}\gf f]\, \big|^q\bigg)^{1/q}
\bigg|L_p(\R)\bigg\|
\\
& \leq  c_6  \,   \bigg\|\bigg( \sum_{j=0}^{\infty}  \big|2^{jtd}\gf ^{-1} [\psi_{j}\gf f]\, \big|^q 
\sum_{ \bar{k} \in \Delta_{j}}2^{(|\bar{k}|_1-jd)t q}\bigg)^{1/q} \bigg|L_p(\R)\bigg\|\, .
\end{split}
  \eeqq
Because of $t>0$  we conclude that
\beqq
\begin{split}
\bigg\|\bigg( \sum_{ \bar{k} \in \N_0^d}   \big|2^{ |\bar{k}|_1t}\gf ^{-1} [\varphi_{\bar{k}}   \psi_{j}\gf f]\, \big|^q\bigg)^{1/q}
\bigg|L_p(\R)\bigg\|  \leq  c_7\,  \|f|F^{td}_{p,q}(\R)\| \, .
\end{split}
\eeqq
From this the claim follows.\\
{\it Step 3.} Let  $0<p,q<\infty$ and $t>\big(\frac{1}{\min(p,q)}-1\big)_+$. We shall proceed by interpolation.
\\
{\it Substep 3.1.} Assume that $\min(p,q)\leq 1$ and $p\leq q$. Since $t>\frac{1}{p}-1$ we choose 
$p_0>1$, $0<\Theta<1$ and $\varepsilon>0$ such that
\beqq
t=\varepsilon + \frac{1}{p}-\frac{1}{p_0}+\frac{\Theta}{p_0}.
\eeqq
Next we define $(p_0,q_0)$, $(p_1,q_1)$  by 
\beqq
 \frac{1}{p}=\frac{1-\Theta}{p_0}+\frac{\Theta}{p_1}\quad \text{and}\quad \frac{p}{q}=\frac{p_0}{q_0}=\frac{p_1}{q_1}.
\eeqq
Now we put $t_0 :=\varepsilon $ and $ t_1:=\frac{1}{\min(p_1,q_1)}+\varepsilon=\frac{1}{p_1}+\varepsilon$ since $p_1\leq q_1$. 
Hence we obtain 
\[
t=(1-\Theta )t_0+\Theta t_1 \qquad \mbox{and}\qquad \frac{1}{q}=\frac{1-\Theta}{q_0}+\frac{\Theta}{q_1}\, .
\]
Proposition \ref{inter2} yields
\beqq
F^{td}_{p,q}(\R)=[F^{t_0d}_{p_0,q_0}(\R),F^{t_1d}_{p_1,q_1}(\R)]_{\Theta}\quad\text{and}\quad  S^{t}_{p,q}F(\R)=[S^{t_0}_{p_0,q_0}F(\R),S^{t_1}_{p_1,q_1}F(\R)]_{\Theta}\,.
\eeqq 
In view of Proposition \ref{inter1},  Steps 1 and 2 we find $F^{td}_{p,q}(\R) \hookrightarrow S^{t}_{p,q}F(\R)$. 
\\
{\it Substep 3.2.} Assume that $\min(p,q)\leq 1$ and $q < p$.
It is enough to interchanges the roles of $p$ and $q$ in Substep 3.1.
\qed
\begin{remark}\label{ga-pe}\rm The interpolation argument in Substep 3.1 does not extend to the case $q_0=q_1=\infty$. It is known that 
 \beqq
[F^{t_0d}_{p_0,\infty}(\R),F^{t_1d}_{p_1,\infty}(\R)]_{\Theta} \not = F^{td}_{p,\infty}(\R)
 \eeqq 
 if $F^{t_0d}_{p_0,\infty}(\R)\not =F^{t_1d}_{p_1,\infty}(\R)$, see \cite{ysy}. 
However, one could apply the $\pm$ method of Gustavsson and Peetre, denoted by $ \langle \, \cdot\,,\, \cdot\,,\Theta \rangle$, to obtain 
  \beqq
 \langle F^{t_0d}_{p_0,\infty}(\R),F^{t_1d}_{p_1,\infty}(\R), \Theta \rangle   = F^{td}_{p,\infty}(\R),
  \eeqq 
  see \cite{ysy}. However, there is no proof of   the assertion 
    \beqq
   \langle S^{t_0}_{p_0,\infty}F(\R),S^{t_1}_{p_1,\infty}F(\R), \Theta \rangle   = S^{t}_{p,\infty}F(\R)
    \eeqq 
availiable in the literature.
\end{remark}

\noindent
{\bf Proof of Theorem \ref{1<p,q<vc-2}.} 
By Theorem \ref{F-SF} and Lemma \ref{q<2-f} it will be enough to deal with $t<0$.
We assume that $F^{td}_{p,q}(\R)\hookrightarrow S^t_{p,q}F(\R)$ if $t<0$. This implies 
$\mathring{F}^{td}_{p,q}(\R)\hookrightarrow \mathring{S}^t_{p,q}F(\R)$ and therefore, by duality, 
$S^{-t}_{p',q'}F(\R)\hookrightarrow F^{-td}_{p',q'}(\R)$. Applying  Example 4 with $a_j :=\delta_{j,\ell}$ we come to a contradiction.
\qed
\\
~~
\\
{\bf Proof of Proposition \ref{pro:2}.} {\it Step 1.} Proof of (i). Since $0<p<1$ we know
$$
[\mathring{S}^t_{p,q}F(\R)]'= S^{-t+\frac{1}{p}-1}_{\infty,\infty}B(\R)\qquad \text{and} \qquad 
[\mathring{F}^{td}_{p,q}(\R)]'= B^{-td+d(\frac{1}{p}-1)}_{\infty,\infty}(\R)\, , 
$$
see Proposition \ref{dual1}.
Assuming $F^{td}_{p,q}(\R)  \hookrightarrow S^{t}_{p,q}F(\R)$ we get
$\mathring{F}^{td}_{p,q}(\R)  \hookrightarrow \mathring{S}^{t}_{p,q}F(\R)$
and therefore
$S^{-t+ \frac 1p -1}_{\infty, \infty} F(\R)  \hookrightarrow B^{-td + d(\frac 1p -1)}_{\infty,\infty}(\R) $.
Since $-td+d(\frac{1}{p}-1)> -t+\frac{1}{p}-1\geq 0$ this is impossible 
(again it will be enough to use $e^{i \bar{k}x}$ as test functions).
Hence $F^{td}_{p,q}(\R) \not \hookrightarrow S^{t}_{p,q}F(\R)$. 
By employing the test function in Example 4 with $a_j:=\delta_{j,\ell}$ we can disprove the 
embedding $S^t_{p,q}F(\R)\hookrightarrow F^{td}_{p,q}(\R)$.
\\
{\it Step 2.} Proof of (ii). We argue  as in the proof of Theorem \ref{main1} replacing $F^t_{p,q}(\R)$ by $F^{td}_{p,q}(\R)$
and taking into account that $t<0$.
The proof is complete. 
\qed


\subsection{Proofs of the optimality assertions}


Let us recall some  results about embeddings of Lizorkin-Triebel spaces.

\begin{lemma}\label{emb}
Let $0 < p < p_0 < \infty$ and $0<q,q_0\leq \infty$.
\begin{enumerate}
\item The embedding $F^{t}_{p,q} (\R) \hookrightarrow F^{t_0}_{p_0,q_0} (\R) $ holds if and only if 
$
 t_0 - \frac{d}{p_0} \leq t- \frac dp.
$
\item The embedding $S^{t}_{p,q} F(\R) \hookrightarrow S^{t_0}_{p_0,q_0} F(\R) $ holds if and only if
$
 t_0 - \frac{1}{p_0} \leq  t- \frac 1p .
$
\end{enumerate}
\end{lemma}

Note that in case $p=p_0$ and $t=t_0$, that is the embedding $F^{t}_{p,q} (\R) \hookrightarrow F^{t}_{p,q_0} (\R) $, 
holds true if and only if 
$q\leq q_0$. A similar statement ist true Lizorkin-Triebel  spaces of dominating mixed smoothness. 
The assertion (i) in Lemma \ref{emb}  can be found in \cite{Ja}, \cite[2.7.1]{Tr83} (sufficiency) and in \cite{sitr} (necessity). 
In case of Triebel-Lizorkin spaces of dominating mixed smoothness 
we refer to \cite{SS} and \cite{HV} (sufficiency). 
Necessity can be traced back to the isotropic case by standard arguments (one considers tensor products of appropriate test functions).
Now we are in position to prove the optimality assertions.\\
\\
{\bf Proof of Theorem \ref{tri-li5}}. 
Assuming $S^{t_0}_{p_0,q_0}F(\R) \hookrightarrow F^{t}_{p,q} (\R) $,  Lemma \ref{p0p1} yields $p_0 \le p$.
Applying Example 2 we derive  
\beqq
 t + 1 - \frac 1p \le t_0 + 1 - \frac{1}{p_0} \Longleftrightarrow  t   - \frac 1p \le t_0  - \frac{1}{p_0}\, .
\eeqq
If $p_0=p$ and $t_0=t$ we use in addition  Example 4. With $a_j:=2^{-jt}$, the embedding $S^{t }_{p ,q_0}F(\R) \hookrightarrow F^{t}_{p,q} (\R) $  
implies that $q_0\leq q$. Altogether Lemma \ref{emb} implies that 
$S^{t_0}_{p_0,q_0}F(\R)  \hookrightarrow S^{t}_{p,q}F(\R)$.
\qed
\\
{~}\\
\noindent
{\bf Proof of Theorem \ref{tri-li2}}. 
Assuming $F^{t_0}_{p_0,q_0} (\R) \hookrightarrow S^{t}_{p,q}F(\R) $
Lemma \ref{p0p1} implies $p_0 \le p$.
Next we employ Example 3. Then the embedding $F^{t_0}_{p_0,q_0} (\R)\hookrightarrow  S^{t}_{p,q}F(\R) $ yields
\[
  d \Big( \frac{t_0}{d} + 1 - \frac{1}{p_0}\Big)  \geq  d \Big(t + 1 - \frac 1p\Big) \qquad 
\Longleftrightarrow \qquad t_0- \frac{d}{p_0}  \geq  d t  - \frac dp \, .
\]
In case $p_0=p$ and $t_0=td$ we use Example 5, again with $a_j:=2^{-jt}$ to obtain $q_0\leq q$. As a consequence of Lemma \ref{emb} 
we arrive at  $F^{t_0}_{p_0,q_0} (\R)  \hookrightarrow F^{td}_{p,q}F(\R) \, . $ \qed
\\
{~}\\
\noindent
{\bf Proof of Theorem \ref{tri-li3}}. 
Assuming $F^{t d}_{p,q} (\R) \hookrightarrow S^{t_0}_{p_0,q_0}F(\R) $
Lemma \ref{p0p1} implies $p \le p_0$. Next we apply Example 3 and  get
\[
  d \Big( t_0 + 1 - \frac{1}{p_0}\Big)  \le  d \Big(t + 1 - \frac 1p\Big) \qquad \Longleftrightarrow 
\qquad t_0- \frac{1}{p_0}  \le   t  - \frac 1p \, .
\]
Working with Example 4 with  $a_j:=2^{-jt}$ we obtain $q\leq q_0$ in case $p=p_0$ and $t=t_0$. 
In a view of Lemma \ref{emb} we conclude that
$S^{t}_{p,q} F(\R) \hookrightarrow S^{t_0}_{p_0,q_0} F(\R)$. 
\qed


\end{document}